\documentclass{article}
\usepackage{amsfonts}
\usepackage{amsmath}
\usepackage{amssymb}
\usepackage{amsthm}
\usepackage{enumerate}
\usepackage[latin1]{inputenc} 
\usepackage{color}
\usepackage{mathtools}
\usepackage{mathdots}
\usepackage{amsmath,amsthm,amssymb}
\usepackage{cite}
\usepackage{authblk}
\usepackage{booktabs}
\usepackage{mathabx}
\usepackage{multirow}
\usepackage{multicol}
\usepackage{tikz}
\usepackage{subdepth}
\usepackage{mathtools}
\usepackage{hyperref}
\newtheorem{theorem}{Theorem}%[section]

\newtheorem{corollary}[theorem]{Corollary}
\newtheorem{lemma}[theorem]{Lemma}
\theoremstyle{definition}

\newtheorem{remark}[theorem]{Remark}
\newtheorem{definition}[theorem]{Definition}

\newcommand{\CC}{{\mathbb C}}

\newcommand{\rank}{{\mbox{\rm rank\,}}}
\newcommand{\diag}{{\mbox{\rm diag}}}

\newcommand{\odd}{\text{\tiny $\mathcal{O}$}}

\newenvironment{smat}{\left[\begin{smallmatrix}}{\end{smallmatrix}\right]}

\usepackage[paperwidth=210mm,paperheight=297mm,textwidth=180mm,textheight=255mm,top=15mm,lmargin=15mm]{geometry}

\title{The equation $X^\top AX=B$ with $B$ skew-symmetric: How much of a bilinear form is skew-symmetric?%\footnote{This work has been supported by the Agencia Estatal de Investigaci\'on os Spain through grants PID2019-106362GB-I00 and MTM2017-90682-REDT}
}\author{Alberto Borobia\thanks{Departamento de Matem\'{a}ticas, Universidad Nacional de Educaci\'on a Distancia (UNED). C/ Senda del Rey, 9, Madrid 28040, Spain. email: {\tt aborobia@mat.uned.es} ORCiD number: 0000-0001-7034-2426}, Roberto Canogar\thanks{Departamento de Matem\'{a}ticas, Universidad Nacional de Educaci\'on a Distancia (UNED). C/ Senda del Rey, 9, Madrid 28040, Spain. email: {\tt rcanogar@mat.uned.es}. ORCiD number: 0000-0002-6952-9311}, and Fernando De Ter\'an\thanks{Departamento de Matem\'{a}ticas, Universidad Carlos III de Madrid. Avda. Universidad 30, 28911, Legan\'es, Spain. email: {\tt fteran@math.uc3m.es}. ORCiD number: 0000-0003-3888-3702. {\bf Corresponding author}.}}
\date{}

\begin{document}

\maketitle

\begin{abstract}
Given a bilinear form on $\CC^n$, represented by a matrix $A\in\CC^{n\times n}$, the problem of finding the largest dimension of a subspace of $\CC^n$ such that the restriction of $A$ to this subspace is a non-degenerate skew-symmetric bilinear form is equivalent to finding the size of the largest invertible skew-symmetric matrix $B$ such that the equation $X^\top AX=B$ is consistent (here $X^\top$ denotes the transpose of the matrix $X$). In this paper, we provide a characterization, by means of a necessary and sufficient condition, for the matrix equation $X^\top AX=B$ to be consistent when $B$ is a skew-symmetric matrix. This condition is valid for most matrices $A\in\mathbb C^{n\times n}$. To be precise, the condition depends on the canonical form for congruence (CFC) of the matrix $A$, which is a direct sum of blocks of three types. The condition is valid for all matrices $A$ except those whose CFC contains blocks, of one of the types, with size smaller than $3$. However, we show that the condition is necessary for all matrices $A$. 
   \end{abstract}
   


\section{Introduction}
Let $A\in\CC^{n\times n}$ represent a bilinear form over $\CC^n$, denoted by ${\mathbb A}:\CC^n\times \CC^n\rightarrow\CC$, and let ${\cal V}$ be an $m$-dimensional subspace of $\CC^n$ (with $m\leq n$). Let $X\in\CC^{n\times m}$ be a matrix whose columns are a basis of ${\cal V}$. Then, the restriction of the bilinear form ${\mathbb A}$ to ${\cal V}$ is non-degenerate and skew-symmetric if and only if the matrix $B:=X^\top AX$ is skew-symmetric and invertible. To see this, consider any pair of vectors of ${\cal V}$, namely $w_1=Xv_1$ and $w_2=Xv_2$, and note that $$
\begin{array}{ccccc}
{\mathbb A}(w_1,w_2)&=&w_1^\top Aw_2&=&v_1^\top Bv_2,\\
{\mathbb A}(w_2,w_1)&=&w_2^\top Aw_1&=&v_2^\top Bv_1,
\end{array}
$$
so if $B^\top=-B$, then ${\mathbb A}(w_1,w_2)=-{\mathbb A}(w_2,w_1)$.

Therefore, the question on determining whether there is an $m$-dimensional subspace of $\CC^n$ such that the restriction of the bilinear form ${\mathbb A}$ to this subspace is non-degenerate and skew-symmetric is equivalent to determining whether there is some $X\in\CC^{n\times m}$ such that $X^\top AX$ is invertible and skew-symmetric. This motivates the interest in analyzing the consistency of the matrix equation
\begin{equation}\label{maineq}
X^\top AX=B,
\end{equation}
with $B$ being invertible and skew-symmetric (and $X$ being the unknown). The main result we obtain in this work is Theorem \ref{sufficiency.th}, that can be roughly stated in the following way:
\begin{quote}
    For most matrices $A\in\CC^{m\times m}$, the equation $X^\top A X=B$, with $B$ skew-symmetric, is consistent if and only if $\rank B\leq 2\rho(A),$
\end{quote}
where $\rho(A)$ is introduced in Definition \ref{components_CFC_def} and depends on the {\em canonical form for congruence} (CFC) of $A$. The statement of Theorem \ref{sufficiency.th} includes a precise description of the meaning of the sentence `for most matrices' (in particular, it refers to those matrices whose 
CFC does not contain some specific blocks). Our second main result is Theorem \ref{necessary_coro}, where we show that the condition $\rank B\leq2\rho(A)$ is a necessary condition for $X^\top AX=B$ to be consistent when $A$ is an arbitrary matrix (without any exceptions), and $B$ is skew-symmetric.

The matrix equation \eqref{maineq} has been of interest since, at least, the 1950's (see, for instance \cite{carlitz1954} and the references in \cite{bcd}). Most of the references on Eq. \eqref{maineq} deal with particular instances of this equation, like the coefficient matrices $A$ and $B$ having some particular structure (symmetric, skew-symmetric, alternating, or bounded rank), and many of them (in particular, the oldest ones) deal with matrices over finite fields. When $A=B$ are complex and invertible, the solutions of Eq. \eqref{maineq} were provided in \cite{wedderburn1921} using the exponential of a matrix. However, up to our knowledge, the question on the consistency of Eq. \eqref{maineq} in the general setting when $A$ is an arbitrary complex matrix had not been addressed until the recent reference \cite{bcd}.  

It is important to emphasize that, if $A$ is skew-symmetric and Eq. \eqref{maineq} is consistent, then $B$ must be skew-symmetric as well. However, the converse is not necessarily true. An elementary counterexample is $A=I_2$ (where $I_k$ stands for the $k\times k$ identity matrix), $B=[0]$ and $X=\left[\begin{smallmatrix}1\\{\mathfrak i}\end{smallmatrix}\right]$ (where $\mathfrak i$ stands for the imaginary unit). 

The present work is a follow-up of \cite{bcd}, where we addressed the solvability of Eq. \eqref{maineq} for the matrix $B$ in the right-hand side being symmetric. The main techniques used in this paper are essentially the ones in \cite{bcd}. In particular, the starting technique is to reduce Eq. \eqref{maineq} to the case where $A$ and $B$ are in CFC, which is a direct sum of blocks (matrices) of three different types (see Theorem \ref{cfc_th}), and then analyzing the equation for $A$ being a direct sum of blocks of just one type, for each of the three types. However, and even though the main results very much resemble the ones in \cite{bcd}, in order to obtain them we have had to spend a considerable effort. This effort has been mainly devoted to solve the corresponding equations that arise in the CFC block-wise decomposition, that have a different right-hand side when $B$ is skew-symmetric compared to the symmetric case. More precisely, when $B$ is symmetric, after reducing to the CFC, the right-hand side is of the form $\left[\begin{smallmatrix}I_k&\\&0_\ell\end{smallmatrix}\right]$ (where $0_\ell$ stands for the null $\ell\times\ell$ matrix). However, when $B$ is skew-symmetric, the CFC of $B$ is of the form $\diag(\left[\begin{smallmatrix}0&1\\-1&0\end{smallmatrix}\right]^{\oplus k},0_\ell)$, where $M^{\oplus k}$ stands for a direct sum of $k$ blocks equal to $M$. A first relevant difference arises, namely: in the case of $B$ being skew-symmetric, the nonzero canonical blocks have size $2\times 2$, whereas, in the symmetric case, they have size $1\times 1$ (they are just blocks of the form $\Gamma_1=1$, see Theorem \ref{cfc_th}). This establishes an additional difficulty in the skew-symmetric case, that forces us to carefully analyze the case where the matrix $A$ in the left-hand side is a direct sum of blocks of each type, and to carry out a relevant effort when plugging in blocks of different types in order to properly accommodate all pieces of the puzzle to get a result valid for any matrix $A$. However, in the case of $B$ being symmetric, it suffices to analyze independently the case where $A$ is just a single block of each type, and from these cases it is possible to plug in any direct sum of blocks (of different types) in a simpler way. Summarizing, obtaining the results in the present paper is far from being a straightforward development of the work carried out in \cite{bcd}.

Among the references mentioned at the beginning, we want to emphasize the early work \cite{carlitz1954-skew}, which addresses a closely related problem to the one in the present paper. In particular, it is obtained in that reference the number of solutions of Eq. \eqref{maineq} where both $A$ and $B$ are skew-symmetric matrices with entries in a finite field. The approach followed in that reference also starts by assuming that both $A$ and $B$ are nonsingular and reducing them to their canonical form for congruence, which is, in both cases, a direct sum of blocks $\left[\begin{smallmatrix}0&1\\-1&0\end{smallmatrix}\right]$, as mentioned above. 

Another related reference to the present work is \cite{ikramov14}, where the solvability of $X^\top DX+AX+X^\top B+C=0$ is considered, with $A,B,C,D$ being complex square matrices of the same size. The interest in equation \eqref{maineq}, or some related equations, has recently increased, mostly related to applications. In particular, the equation $X^*AX=B$, with $A,B$ being skew-Hermitian with zero trace, and $X$ invertible, has been considered quite recently in \cite{benzi-viviani}, whereas the equation $DX+X^\top A-X^\top BX+C=0$, with $A,B,C,D\in\mathbb R^{n\times n}$, has been addressed in \cite{bimp,benner-palitta}.

We want to emphasize that we are not imposing any restriction on the solution $X$. In particular, $X$ is not necessarily square and invertible and, actually, the interesting case is, precisely, the one where it is not invertible, since there is an invertible solution to Eq. \eqref{maineq} if and only if $A$ and $B$ have the same CFC (that is, they are {\em congruent}). 

The paper is organized as follows. In Section \ref{basic.sec} we introduce the basic notation and tools that are used throughout the manuscript, and in Section \ref{rho_sec} we introduce the quantity $\rho(A)$, for an arbitrary matrix $A$, which is key in the developments and the characterization for consistency of Eq. \eqref{maineq}. In Section \ref{necessary_sec} we present a necessary condition for Eq. \eqref{maineq} to be consistent when $B$ is skew-symmetric, valid for an arbitrary matrix $A$. Sections \ref{JntoB_sec} and \ref{typeII_sec} are devoted to prove that this condition is also sufficient when the CFC of $A$ contains only blocks of either  Type-0 or Type-II, respectively. As for the Type-I blocks, we prove in Section \ref{typeI_sec} that, if the CFC of $A$ contains only Type-I blocks with size at least $3$, then the condition is sufficient as well. Actually, we also show that, when Type-I blocks with size either $1\times 1$ or $2\times2$ appear, the condition is not sufficient anymore. By putting all pieces from Sections \ref{JntoB_sec}, \ref{typeII_sec}, and \ref{typeI_sec} together, we prove, in Section \ref{sufficiency_sec}, that the condition obtained in Section \ref{necessary_sec} is sufficient when $A$ is a matrix whose CFC does not contain Type-I blocks with either size $1\times1$ or $2\times2$ (but it can be any direct sum, with size $n\times n$, of any other blocks). In Section \ref{generic.sec} we particularize the condition for the solution of Eq. \eqref{maineq} to the case of a generic matrix $A$ (namely, the one having a generic CFC), and we obtain a characterization for Eq. \eqref{maineq} to be consistent when $B$ is skew-symmetric for most matrices. Finally, Section \ref{conclusion.sec} presents the main conclusions of this work, and some lines of further research.

\section{Basic setting}\label{basic.sec}

As we have mentioned in the Introduction, the Canonical Form for Congruence (CFC) of matrices plays a central role in this work. In order to recall the CFC we first need to introduce the following matrices:
$$
J_k(\lambda):=\left[\begin{array}{c@{\mskip8mu}c@{\mskip8mu}c@{\mskip8mu}c}
\lambda&1\\[-4pt]&\ddots&\ddots\\[-4pt]&&\lambda&1\\[-2pt]&&&\lambda\end{array}\right]
$$
is a $k\times k$ {\em Jordan block associated with
$\lambda\in\CC$}; for each $k\geq1$, let $\Gamma_k$ be the $k\times k$ matrix
$$
\Gamma_k:=\left[\begin{array}{c@{\mskip8mu}c@{\mskip8mu}c@{\mskip8mu}c@{\mskip8mu}c@{\mskip8mu}c}0&&&&&(-1)^{k+1}\\[-4pt]
&&&&\iddots&(-1)^k\\[-4pt]&&&-1&\iddots&\\&&1&1&&\\&-1&-1&&&\\1&1&&&&0\end{array}\right]\
\qquad (\Gamma_1=[1]);
$$
and, for each $\lambda\in\CC$ and each $k\geq1$, $H_{2k}(\lambda)$ is the $2k\times2k$ matrix
$$
    H_{2k}(\lambda):=\begin{bmatrix}
    0&I_k\\
     J_k(\lambda)&0
    \end{bmatrix}.
$$
Note, in particular, that $H_2(-1)=\left[\begin{smallmatrix}0&1\\-1&0\end{smallmatrix}\right]$. This block will be key in our developments, since the CFC of any invertible skew-symmetric matrix is a direct sum of blocks of this form.

Now, we state the basic tool in our developments, namely the CFC form introduced in \cite{hs2006} (see also \cite{semaj} for some history on this canonical form).

\begin{theorem}\label{cfc_th}{\rm (Canonical form for congruence, CFC) \cite[Th. 1.1]{hs2006}.}
Each square complex matrix is congruent to a direct sum, uniquely determined up to permutation of addends, of canonical matrices of the following three types
\begin{center}
\begin{tabular}{|c|c|}\hline
     Type 0& $J_k(0)$ \\\hline
     Type I& $\Gamma_k$\\\hline
     Type II&\begin{tabular}{c}$H_{2k}(\mu)$,\\ \footnotesize$0 \neq\mu\neq(-1)^{k+1}$\\
\footnotesize($\mu$ is determined up to replacement by $\mu^{-1}$)\end{tabular}\\\hline
\end{tabular}
\end{center}
\end{theorem}

\subsection{Notation}

We will use the following notation in the manuscript:

\begin{enumerate}
    \item Given two complex matrices $A\in\CC^{n\times n}$ and  $B\in\CC^{m\times m}$, and an unknown matrix $X$ of size $n\times m$, then:

\begin{enumerate}[(i)]
    \item $A\rightsquigarrow B$ means that the equation $X^\top A X= B$ is consistent, that is, $X_0^\top A X_0= B$ for some  $X_0\in\CC^{n\times m}$. 

    \item $A \overset{X_0}{\rightsquigarrow} B$ means that $X_0^\top A X_0= B$. In other words, the equation $X^\top A X= B$ is consistent and $X_0$ is a  solution. 

    \item  $A \rightsquigarrow\hspace{-3ex}/ \quad B$ means that the equation $X^\top A X= B$ is not consistent. 
    
\end{enumerate}

\item $A^{\oplus k}=A\oplus \cdots \oplus A$ is the direct sum of  $k$ copies of $A$.

\item Throughout the manuscript, some matrices with null size, like  $A^{\oplus 0}$, $J_0(0)$, $ H_0(\mu)$, $\Gamma_0$, $I_0$, or  $0_{0\times 0}$, will appear. All of them are understood as empty matrices.

\item Sometimes we will write CFC($A$) to denote the CFC of the matrix $A$.

\end{enumerate}

The notation $A\rightsquigarrow B$ emphasizes the fact that the equation $X^\top A X=B$ being consistent establishes a relation between $A$ and $B$ (in this order), for $A\in\CC^{n\times n}$ and $B\in\CC^{m\times m}$.

For brevity, we will sometimes say that a matrix $M$ has {\em size $n$} when $M$ is $n\times n$.

\subsection{The basic laws of consistency} \label{thelaws}

Here we present several basic results that will simplify our further developments. 

\begin{lemma} \label{CFC-Skew_lemma}
\begin{enumerate}[{\rm(i)}]
    \item \label{skewsym_remain}
    If $A$ is a skew-symmetric matrix and  $A \rightsquigarrow B$, then $B$ is skew-symmetric.

    \item \label{CFC-Skew_2}
    If $A$ is an $n\times n$ skew-symmetric matrix, then {\rm CFC}$(A)=H_2(-1)^{\oplus k}\oplus J_1(0)^{\oplus n-2k}$ for some $k\leq n/2$.

\end{enumerate}
\end{lemma}
\begin{proof} 
\begin{enumerate}[(i)]
    \item If $A$ is skew-symmetric and $X_0$ is a solution of $X_0^\top A X_0=B$, then $$B^\top=X_0^\top A^\top X_0=-X_0^\top A X_0=-B,$$ so $B$ is also skew-symmetric.
    \item The only skew-symmetric canonical matrices in Theorem~\ref{cfc_th} are $H_2(-1)$ and $J_1(0)$. Since CFC($A$) is congruent to $A$, and $A$ is skew-symmetric, then CFC($A$) is skew-symmetric as well, so it must be a direct sum containing only blocks of these two kinds.
\end{enumerate}
\end{proof}

\begin{remark}\label{evenrank_rem}
A consequence of Lemma \ref{CFC-Skew_lemma} (ii) is the well-known fact that the rank of a skew-symmetric matrix is an even number.
\end{remark}

The following result includes some basic laws of consistency that are already known and/or straightforward to check. 

\begin{lemma} \label{basic_laws_lemma}
\label{directsum.lemma} \label{TransitivityOfConsistency_lem} \label{suma_directa_rightsquigarrows_lem} 
{\bf Laws of consistency:} For any complex square matrices $A,B,C,A_i,B_i$, the following properties hold:
\begin{enumerate}[{\rm(i)}]
    \item {\bf Addition law.} \label{additive_law}
    If $A_i \overset{X_i}{\rightsquigarrow} B_i$, for $1\leq i\leq k$, then  $\bigoplus_{i=1}^{k} A_i \overset{X}{\rightsquigarrow} \bigoplus_{i=1}^{k} B_i$, with $X=\bigoplus_{i=1}^{k} X_i$.
    
    \item {\bf Transitivity law.} {\rm(\cite[Lemma 2.4]{bcd})}. \label{transitive_law}
    If $A \overset{X_0}{\rightsquigarrow} B$ and $B \overset{Y_0}{\rightsquigarrow} C$, then $A \overset{X_0Y_0}{\rightsquigarrow} C$.  
    
\item {\bf Permutation law.} \label{perm_law} $\bigoplus_{i=1}^\ell A_i\rightsquigarrow \bigoplus_{i=1}^\ell A_{\sigma(i)}$, for any permutation $\sigma$ of $\{1,\hdots,\ell\}$.
    
    \item {\bf Elimination law.}\label{elim_law} 
    $A \oplus B \overset{X_0}{\rightsquigarrow} A$, with $X_0=\left[\begin{smallmatrix}I_n\\ 0\end{smallmatrix}\right]$, and where $n$ is the size of $A$.

    \item{\bf Canonical reduction law.}\label{reduction_law} {\rm(\cite[\S 2]{bcd})}. If $A$ and $B$ are congruent to, respectively, $\widetilde A$ and $\widetilde B$, then $A\rightsquigarrow B$ if and only if $\widetilde A\rightsquigarrow\widetilde B$.
    
    \item {\bf $J_1(0)$-law.} {\rm(\cite[Lemma 2.2]{bcd})}. \label{zeroes_law}
    For $k,\ell\geq 0$ we have  $A\oplus J_1(0)^{\oplus k} \rightsquigarrow  B \oplus J_1(0)^{\oplus \ell}$ if and only if  $A \rightsquigarrow  B$.

\end{enumerate}
\end{lemma}

The Addition, Transitivity, Permutation, and Elimination laws (parts \eqref{additive_law}--\eqref{elim_law} of Lemma \ref{basic_laws_lemma}) will be sometimes used without any explicit reference to them. In particular, the Permutation law will allow us to reorder the canonical blocks of the CFC of $A$ without affecting at all the consistency.

Notice that the Canonical reduction law, together with the $J_1(0)$-law of consistency,  guarantee that, when looking for the consistency of Eq.~\eqref{maineq}, there is no loss of generality in assuming that $A$ and $B$ are given in CFC, that $A$ and $B$ have no blocks of type $J_1(0)$, and that $B$ is the direct sum of blocks of type $H_2(-1)$ (which is the CFC of a skew-symmetric invertible matrix). For the sake of completeness, in the main results of this work we have considered that CFC($A$) may contain blocks of type $J_1(0)$. However, in the proofs we will assume that $A$ does not contain blocks of this form, according to the $J_1(0)$-law of consistency.

\section{The quantity $\rho(A)$}\label{rho_sec}

The main result of this work (Theorem \ref{sufficiency.th}) depends on an intrinsic quantity of the matrix $A$, that we denote by $\rho(A)$. In this section, we introduce it and show its main properties in the context of this work.

\begin{definition} \label{components_CFC_def}
Let $A\in\CC^{n\times n}$ whose {\rm CFC} contains, exactly 
\begin{enumerate}[\rm(i)]
\item\label{type01} $j_1$  Type-$0$ blocks with size $1$;

\item  $j_{\odd}$ Type-$0$ blocks with odd size at least 3;

\item  $\gamma_{\varepsilon}$ Type-I blocks with even size;

\item  $h^-_{2\odd}$ Type-II blocks of the form $H_{4k-2}(-1)$, for any $k\geq 1$; and

\item\label{othertypes} an arbitrary number of other Type-$0$ and Type-II blocks.
\end{enumerate}
Then we define the quantity
\begin{equation} \label{rho(A)_formula}
    \rho(A):= \frac{n-j_1+j_{\odd}+\gamma_{\varepsilon}+2h^-_{2\odd}}{4}
\end{equation}
which satisfies the following essential, and straightforward to prove, additive  property:
\begin{equation} \label{additive}\rho(A_1\oplus\cdots \oplus A_k)=\rho(A_1)+\cdots+\rho(A_k),
\end{equation}
for any square complex matrices $A_1,\ldots, A_k$.
\end{definition}

\begin{remark}
Since the quantities in Definition \ref{components_CFC_def} will appear many times throughout the manuscript, we have tried to find an adequate notation to help the reader to identify them. More precisely, the letters for the number of blocks in parts (i)--(vi) recall the notation for the corresponding blocks. In particular, $j$ is the initial lowercase letter of ``Jordan", since Type-0 blocks are Jordan blocks, $\gamma$ is the lowercase of $\Gamma$, which is the letter used for Type-I blocks, and $h$ is also the lowercase of $H$, which is the letter used for Type-II blocks. As for the subindices, $\varepsilon$ stands for ``even" and $\odd$ stands for ``odd".
\end{remark}

\begin{remark} \label{rho_values}
It is convenient to keep in mind the value of $\rho$ for each canonical block: 
\begin{enumerate}[(i)]

\item  $\rho\left(J_1(0)\right)=0$; 

\item  $\rho(J_{2k-1}(0))=\frac{k}{2}$ when   $k\geq 2$;

    \item  $\rho(J_{2k}(0))=\frac{k}{2}$ when   $k\geq 1$;

\item $\rho(\Gamma_{2k-1})=\frac{2k-1}{4}$ when $k\geq 1$;

\item  $\rho(\Gamma_{2k})=\frac{2k+1}{4}$ when $k\geq 1$;

\item $\rho(H_{4k-2}(-1))=k$ when $k\geq 1$;  

\item $\rho(H_{4k}(1))=k$ when $k\geq 1$;

\item $\rho(H_{2k}(\mu))=\frac{k}{2}$ when   $k\geq 1$ and $\mu \neq \pm1,0$.

\end{enumerate}
\end{remark}

\section{A necessary condition for consistency} \label{necessary_sec}

This section is devoted to state and prove the following necessary condition for the consistency of Eq.~(\ref{maineq}) with $B$ skew symmetric and invertible. 

\begin{lemma}\label{necessary_th}
If $A$ is a complex square matrix such that $X^\top AX= H_2(-1)^{\oplus m}$ is consistent, then $m \leq  \rho(A)$.
\end{lemma}
\begin{proof} We may assume that $A$ is an  $n\times n$ complex matrix  given in CFC and, by the $J_1(0)$-law of consistency of Lemma \ref{basic_laws_lemma}, also that $A$ has no  blocks of type $J_1(0)$ (see the last paragraph of Section~\ref{thelaws}).

For $A$ let $j_1, j_\odd, \gamma_\varepsilon, h_{2 \odd}^{-}$ be as in Definition \ref{components_CFC_def}, and let $j_{\varepsilon}$ be the number of all Type-0 blocks with even size, let $\gamma_{\odd}$ be the  number of Type-I blocks with odd size, $h_2$ be the number of Type-II blocks of the form $H_{2k}(\mu_k)$ with $\mu_k\neq(-1)^{k}$, and $h^+_{2\varepsilon}$ be the number of Type-II blocks of the form $H_{4k}(1)$, for any $k\geq1$. Then, we can write $A$ as 
\begin{eqnarray*} 
A & =&  \left(\bigoplus_{i=1}^{j_\odd} J_{2m_i-1}(0)\right)\oplus
\left(\bigoplus_{i=j_\odd+1}^{j_\odd+j_\varepsilon} J_{2m_i}(0)\right)\oplus
\left(\bigoplus_{j=1}^{\gamma_\odd} \Gamma_{2 \widehat m_j-1}\right)\oplus
\left(\bigoplus_{j=\gamma_\odd+1}^{\gamma_\odd+\gamma_\varepsilon} \Gamma_{2\widehat m_j}\right)\\
&&\oplus 
\left(\bigoplus_{k=1}^{h_2} H_{2 \widecheck m_k} (\mu_k)\right)\oplus
\left(\bigoplus_{k=h_2+1}^{h_2+h_{2\varepsilon}^+} H_{4 \widecheck m_k} (1)\right)\oplus
\left(\bigoplus_{k=h_2+h_{2\varepsilon}^++1}^{h_2+h_{2\varepsilon}^++h_{2\odd}^-} H_{4 \widecheck m_k-2} (-1)\right),
\end{eqnarray*}
where $m_i>1$, for $i=1,\ldots,j_\odd$; $m_i\geq 1$, for $i=j_\odd+1,\ldots,j_\odd+j_\varepsilon$; $\widehat m_j\geq 1$, for $j=1,\hdots,\gamma_\odd+\gamma_\varepsilon$;  $\widecheck m_k\geq 1$, for $k=1,\hdots,h_2+h_{2\varepsilon}^++h_{2\odd}^-$; and where $\mu_k\neq 0,\pm1$, for $k=1,\hdots,h_2$. 
Let us introduce the notation
$$
\begin{array}{ccl}
M&=&m_1+\cdots+m_{j_\odd+j_\varepsilon},\\
\widehat M&=&\widehat m_1+\cdots+\widehat m_{\gamma_\odd+\gamma_\varepsilon},\\
\widecheck M&=&\widecheck m_1+\cdots+\widecheck m_{h_2}+2\widecheck m_{h_2+1}+\cdots+2\widecheck m_{h_2+h_{2\varepsilon}^+}+2\widecheck m_{h_2+h_{2\varepsilon}^++1}+\cdots+2\widecheck m_{h_2+h_{2\varepsilon}^++h_{2\odd}^-}.
\end{array}
$$
Then, we will show that
\begin{eqnarray}
n&=&2(M+\widehat M+\widecheck M)-j_\odd-\gamma_\odd-2h_{2\odd}^-,\label{n3}\\
\rank (A+A^\top)&=&2(M+\widehat M+\widecheck M)-2j_\odd-\gamma_\odd-\gamma_\varepsilon-4h_{2\odd}^-. \label{rankA+AT3}
\end{eqnarray}

Equation \eqref{n3} is immediate. In order to get \eqref{rankA+AT3}, it suffices to check that:
\begin{eqnarray}
\rank( J_{n}(0)+J_n(0)^\top)&=&\left\{\begin{array}{cc}n-1&\mbox{if $n$ is odd}, \\ n&\mbox{if $n$ is even},\end{array}\right.\label{rankJ}\\
\rank (\Gamma_n+\Gamma_n^\top)&=&\left\{\begin{array}{cc}n&\mbox{if $n$ is odd},\\
n-1&\mbox{if $n$ is even},\end{array}\right.\label{rankgamma}\\
\rank (H_{2n}(\mu)+H_{2n}(\mu)^\top)&=&2n \quad \mbox{if $\mu\neq0,1,-1$},\label{rankH} \\
\rank (H_{4n}(1)+H_{4n}(1)^\top)&=&4n,\label{rankH4} \\
\rank (H_{4n-2}(-1)+H_{4n-2}(-1)^\top)&=&4n-4.\label{rankH4-2} 
\end{eqnarray}

To get the second identity in Equation \eqref{rankJ} (i.e. for $n$ even) we can prove that $\det(J_{n}(0)+J_n(0)^\top)=\pm1$ when $n$ is even. This can be done by induction, spanning the determinant across the first row, then across the first column, and then using induction in the remaining minor, namely $\det(J_{n-2}(0)+J_{n-2}(0)^\top)$, together with the identity $\det(J_{2}(0)+J_2(0)^\top)=-1$. As for the first identity in Equation \eqref{rankJ} (i.e. for $n$ odd), we can also use induction in the same way to prove that  $\det(J_{n}(0)+J_n(0)^\top)=0$, since $\det(J_1(0)+J_1(0)^\top)=0$. Finally, note that $\rank( J_{n}(0)+J_n(0)^\top)$ must be at least $n-1$, since
$$
J_n(0)+J_n(0)^\top=\left[\begin{array}{c|c}0&e_1^\top\\\hline e_1&J_{n-1}(0)+J_{n-1}(0)^\top\end{array}\right],
$$
and the $2\times2$ block in the partition above has rank $n-1$, by the previous arguments.

Equation \eqref{rankgamma} is a consequence of the identity
$$
\Gamma_n+\Gamma_n^\top=\left\{\begin{array}{cc} \vspace{2mm}
\left[\begin{smallmatrix}&&&&0\\&&&\iddots&2\\&&0&\iddots&\\&0&-2&&\\0&2&&\end{smallmatrix}\right]&\mbox{if $n$ is even},\\
\left[\begin{smallmatrix}&&&&2\\&&&-2\\&&\iddots&&&\\&-2&\\2&&\end{smallmatrix}\right]&\mbox{if $n$ is odd}.\\
\end{array}\right.
$$
Finally, Equations \eqref{rankH}--\eqref{rankH4-2} follow from the identity:
$$
H_{2n}(\mu)+H_{2n}(\mu)^\top=\footnotesize{\left[\begin{array}{cccc|cccc}
&&&&1+\mu&\\
&&&&1&\ddots\\
&&&&&\ddots&1+\mu\\
&&&&&&1&1+\mu\\ \hline
\mu+1&1&&&&&&\\
&\ddots&\ddots&&&&&\\
&&\mu+1&1&&&&\\
&&&\mu+1&&&&
\end{array}\right].}
$$
From \eqref{n3}--\eqref{rankA+AT3} we conclude
\begin{equation}\label{n-rank(A+AT)}
    n-\rank(A+A^\top)=j_\odd+\gamma_\varepsilon+2h_{2\odd}^-.
\end{equation}

By hypothesis   $A\rightsquigarrow H_2(-1)^{\oplus m}$. This implies that there exist some $X_0 \in\CC^{n\times 2m}$ such that 
\begin{equation} \label{B=skew}
    X_0^\top A X_0=H_2(-1)^{\oplus m}.
\end{equation}
Now, adding \eqref{B=skew} with its transpose, we get
\begin{equation}\label{transpose}
X_0^\top(A+A^\top)X_0=0.
\end{equation}
From \eqref{B=skew} it follows that  $\rank X_0=\rank X_0^\top=2m$. Using this fact, together with the well-know inequality (see \cite[page 13]{horn-johnson}) 
\begin{equation} \label{RankAB2}
\rank(PQ)\geq \rank P + \rank Q -k, \qquad \text{for } P\in \mathbb{C}^{p\times k} \text{ and }   Q\in \mathbb{C}^{k\times q},
\end{equation}
we obtain 
$$
\rank(X_0^\top(A+A^\top))\geq 2m+\rank(A+A^\top)-n,
$$
and, then, using again \eqref{RankAB2} with $P=X_0^\top (A+A^\top)$ and $Q=X_0$ we obtain
$$
0=\rank(X_0^\top(A+A^\top)X_0)\geq \big(2m+\rank(A+A^\top)-n\big)+\rank X_0-n=4m-j_\odd-\gamma_\varepsilon-2h_{2\odd}^{-}-n,
$$
so
\begin{equation*}
m\leq \frac{n+j_\odd+\gamma_\varepsilon+2h_{2\odd}^{-}}{4}=\rho(A).
\end{equation*}

\end{proof}

In Lemma \ref{necessary_th}, the right-hand side matrix $H_2(-1)^{\oplus m}$ is invertible. We can include also singular matrices using the $J_1(0)$-law (Lemma \ref{basic_laws_lemma} \eqref{zeroes_law}). As a consequence, we obtain the following result.

\begin{theorem}\label{necessary_coro}
Let  $A$ and $B$ be complex square matrices with $B$  skew-symmetric. If  $X^\top AX= B$, for some $X$, then $\rank(B)\leq  2 \rho(A)$.
\end{theorem}

We want to emphasize that Theorem~\ref{necessary_coro} provides a necessary condition for Eq.~\eqref{maineq} to be consistent, when $B$ is skew-symmetric, that is valid for any square matrix $A$. To see this, consider Lemma~\ref{CFC-Skew_lemma} \eqref{CFC-Skew_2} and  note that conditions (i)--(vi) in Definition~\ref{components_CFC_def} are not restrictive at all, but just a particular description of the CFC of an arbitrary matrix $A$. The rest of the paper is aimed to prove that if CFC($A$) does not have  either blocks $\Gamma_1$ or  $\Gamma_2$ then $m\leq\rho(A)$ is also a sufficient condition for Eq.~\eqref{maineq} to be consistent. The reason for excluding these blocks is provided in Section \ref{type2blocksbad.sec}.

The main strategy to prove that $m\leq \rho(A)$ is also sufficient is the use of ``$\rho$-invariant" relations, namely relations of the form $A\rightsquigarrow B$ such that $\rho(A)=\rho(B)$. We want to emphasize that all relations $\rightsquigarrow$ in this work, except those corresponding to the Elimination law, are $\rho$-invariant.

The Canonical reduction law (Lemma \ref{basic_laws_lemma} \eqref{reduction_law}) allows us to start from CFC($A$), which can be written as a direct sum  $A_0\oplus A_1\oplus A_2$ where, $A_0, A_1$, and $A_2$ are, respectively, the direct sum of all Type-0 blocks, all Type-I blocks, and all Type-II blocks. In the first place we prove, in Sections \ref{JntoB_sec}, \ref{typeII_sec}, and \ref{typeI_sec}, that, for $i=0,1,2$, the condition $m_i\leq\rho(A_i)$ is sufficient. In order to do this, we concatenate an appropriate series of $\rho$-invariant relations to get $H_2(-1)^{\oplus\lfloor\rho(A_i)\rfloor}\,\oplus C_i$ (where the leftover $C_i$ could be the empty matrix) and such that $\rho(C_i)=\rho(A_i)-\lfloor\rho(A_i)\rfloor< 1$. Using the Addition and Transitivity laws (Lemma \ref{basic_laws_lemma} \eqref{additive_law}--\eqref{transitive_law}) in the previous relations we conclude that 
\begin{equation}\label{first_step}
A\rightsquigarrow H_2(-1)^{\oplus\lfloor\rho(A_0)\rfloor+\lfloor\rho(A_1)\rfloor+\lfloor\rho(A_2)\rfloor}\oplus C_0 \oplus C_1\oplus C_2.
\end{equation} 
Then, in Section \ref{sufficiency_sec}, we will show, for the leftovers, that
\begin{equation}\label{second_step}
C_0 \oplus C_1\oplus C_2 \rightsquigarrow H_2(-1)^{\oplus \lfloor\rho(C_0)+\rho(C_1)+\rho(C_2)\rfloor}.
\end{equation} 
As a conclusion of Equations (\ref{first_step}) and (\ref{second_step}), $m\leq\rho(A)$ is also sufficient when CFC($A$) has blocks of any type, except $\Gamma_1$ and $\Gamma_2$.

\section{The case where CFC($A$) is a direct sum of Type-0 blocks}\label{JntoB_sec}

In Lemmas \ref{Js_lem} and \ref{type0_lem} we provide a collection of  auxiliary statements about the consistency of Equation \eqref{maineq}. The proof of Lemma \ref{Js_lem} is straightforward, since it reduces to check that the provided matrix $X_0$ is indeed a solution of the corresponding equation.

\begin{lemma}\label{Js_lem} 
The following consistency statements hold:
\begin{enumerate}[{\rm(i)}]
\item \label{Jk+tToJk} 
$J_{2k}(0) \overset{X_0}{\rightsquigarrow} J_{2k-1}(0)$, for  $k \geq 1$  and   $X_0=\begin{smat} 0_{1\times 2k-1} \\ I_{2k-1}  \end{smat}$. 
\item \label{Jk+4ToJk}
$ J_{k+4}(0) \overset{X_0}{\rightsquigarrow} J_k(0) \oplus J_{3}(0)$, for  $k \geq  1$  and   $X_0=I_k\oplus \begin{smat}
 0_{1\times 3} \\ I_3
\end{smat}$. 
\item  \label{J2^2toH2(-1)} 
$ J_2(0)^{\oplus 2} \overset{X_0}{\rightsquigarrow} H_2(-1)$, for 
$X_0=\begin{smat} 1 & 0 \\ 0 & 1 \\ 0 & 1 \\ -1 & 0\end{smat}$. 
\item\label{J3toH2(-1)} $ J_3(0) \overset{X_0}{\rightsquigarrow} H_2(-1)$, for 
$X_0= \begin{smat} 1 & 0 \\  0 & 1 \\  -1 & 0 \end{smat}$.
\item\label{J5toH2(-1)} $ J_5(0) \overset{X_0}{\rightsquigarrow} H_2(-1)\oplus J_2(0)$, for 
$X_0=\begin{smat} 0 & 1 & 1 & 0 \\ 
0 & 0 & 0 & 1 \\
0 & -1 & 0 & 0 \\
1 & 0 & 0 & 1 \\
0 & 1 & 0 & 0 \\ \end{smat}$. \label{J5toH2+J2}
\end{enumerate}
\end{lemma}

\begin{lemma}\label{type0_lem} 
For $m\geq1$, the following consistency statements hold:
\begin{enumerate}[{\rm(a)}]
   \item  $J_{4m-1}(0) \rightsquigarrow H_2(-1)^{\oplus m}\label{type0_4m-1}$ and $\rho(J_{4m-1}(0))=m$.
    \item $J_{4m}(0) \rightsquigarrow H_2(-1)^{\oplus m}\label{eq_type0_4m}$ and $\rho(J_{4m}(0))=m$.
    \item $J_{4m+1}(0)\rightsquigarrow H_2(-1)^{\oplus m}\oplus J_2(0)\label{type0_4m+1}$ and $\rho(J_{4m+1}(0))=m+\frac12$.
    \item $J_{4m+2}(0)\rightsquigarrow H_2(-1)^{\oplus m} \oplus J_2(0)$ and $\rho(J_{4m+2}(0))=m+\frac12$.
\end{enumerate}
\end{lemma}

\begin{proof} The second claim in each item is given in Remark~\ref{rho_values}. Let us prove the first claims.
\begin{enumerate}[(a)]
\item For $m=1$, the statement is Lemma~\ref{Js_lem}~\eqref{J3toH2(-1)}. For $m>1$ we have 
$$
J_{4m-1}(0) \rightsquigarrow J_{4m-5}\oplus J_3(0) \rightsquigarrow \cdots \rightsquigarrow J_3(0)^{\oplus m} \rightsquigarrow H_2(-1)^{\oplus m},
$$
where all except the last step are due to Lemma~\ref{Js_lem}~\eqref{Jk+4ToJk}, and the last step is due to Lemma~\ref{Js_lem}~\eqref{J3toH2(-1)}. Note that, at each step (except the first one), we also use the Addition law (Lemma \ref{basic_laws_lemma} \eqref{additive_law}).

 \item Using Lemma~\ref{Js_lem}~\eqref{Jk+tToJk} and part \eqref{type0_4m-1} of the present lemma, we have:  
$J_{4m}(0) \rightsquigarrow  J_{4m-1}(0) \rightsquigarrow  H_2(-1)^{\oplus m}.$ 

\item  For $m=1$, the claim is Lemma~\ref{Js_lem}~\eqref{J5toH2(-1)}.  For $m>1$ we have 
$$
J_{4m+1}(0) \rightsquigarrow J_{4m-3}\oplus J_3(0) \rightsquigarrow \cdots \rightsquigarrow J_5(0)\oplus J_3(0)^{\oplus m-1} \rightsquigarrow H_2(-1)^{\oplus m} \oplus J_2(0),
$$
where all except the last step are due to Lemma~\ref{Js_lem}~(\ref{Jk+4ToJk}) and the last step is due to Lemma~\ref{Js_lem}~(\ref{J3toH2(-1)})--(\ref{J5toH2+J2}). Note that, again, we use the Addition law.

\item Using Lemma~\ref{Js_lem}~(\ref{Jk+tToJk}) and part \eqref{type0_4m+1} in the present lemma, we have:    
$J_{4m+2}(0)\rightsquigarrow J_{4m+1}(0)\rightsquigarrow H_2(-1)^{\oplus m}\oplus J_2(0).$
\end{enumerate}
\end{proof}

Now we are ready to prove the converse of Lemma~\ref{necessary_th} when only Type-0 blocks are involved in the CFC of $A$. In other words, if $A$ is a complex square matrix  whose {\rm CFC} has only Type-0 blocks and  $m \leq  \rho(A)$  then  $A \rightsquigarrow H_2(-1)^{\oplus m}$. This is a direct consequence of Lemma~\ref{directsum.lemma} and the following result.

\begin{theorem}\label{Js_th}
For a complex square matrix $A$   whose {\rm CFC} has only Type-0 blocks, one of the following is satisfied: 
\begin{enumerate}[{\rm1.}]
\item\label{rhointeger} $\rho(A)$ is an integer and  $ A\rightsquigarrow H_2(-1)^{\oplus \rho(A)}$.
\item\label{rho-1/2integer}  $\rho(A)-\frac{1}{2}$ is an integer and   $A \rightsquigarrow H_2(-1)^{\oplus \rho(A)-\frac{1}{2}}\oplus J_2(0)$. 
\end{enumerate}
\end{theorem}

\begin{proof} We may assume  that $A$ is in CFC. If $\widetilde A$ is obtained from $A$ by deleting the blocks $J_1(0)$ then $\rho(\widetilde A)=\rho(A)$. This fact, together with Lemma~\ref{basic_laws_lemma}~(\ref{zeroes_law}), permits us to assume that $A$ has no  blocks of type $J_1(0)$.  So we can write $A$ as 
\begin{eqnarray*} 
A= \left(\bigoplus_{h=1}^{\alpha} J_{4 \widetilde   m_h-1} (0)\right)\oplus  
\left(\bigoplus_{i=1}^{\beta} J_{4  \widehat m_i} (0)\right)\oplus 
\left(\bigoplus_{j=1}^{\gamma} J_{4  \widecheck m_j +1}(0) \right)\oplus 
\left(\bigoplus_{k=1}^{\delta} J_{4 m'_k+2}(0) \right),
\end{eqnarray*}
where $\alpha, \beta,\gamma,\delta \geq 0$, and
$$\widetilde m_1, \ldots, \widetilde m_\alpha\geq 1; \
\phantom{m=}\widehat m_1, \ldots, \widehat m_\beta\geq 1; \ 
\phantom{m=}\widecheck m_1,\ldots, \widecheck m_\gamma\geq 1; \ 
\phantom{m=}m'_1,\ldots,  m'_\delta\geq 0.$$

Let us define also
$$\widetilde m= \widetilde m_1+\cdots+ \widetilde m_\alpha; \quad 
\widehat m= \widehat m_1+\cdots+ \widehat m_\beta; \quad 
\widecheck m= \widecheck m_1+\cdots+ \widecheck m_\gamma; \quad 
m'=  m'_1+\cdots+  m'_\delta.$$

From the additive property of $\rho$ (namely, \eqref{additive}) and Lemma~\ref{type0_lem} it follows that 
\begin{eqnarray} \label{rho(A)}
\rho(A)& = & 
\sum_{h=1}^{\alpha}\rho(J_{4 \widetilde m_h-1} (0))+
\sum_{i=1}^{\beta}\rho(J_{4 \widehat m_i} (0))+
\sum_{j=1}^{\gamma}\rho(J_{4 \widecheck m_j+1} (0))+
\sum_{k=1}^{\delta}\rho(J_{4  m'_k+2} (0))\\ \nonumber
&=& 
\sum_{h=1}^{\alpha}\widetilde m_h+
\sum_{i=1}^{\beta}\widehat m_i+
\sum_{j=1}^{\gamma}\left(\widecheck m_j+\frac12\right)+
\sum_{k=1}^{\delta}\left(m'_{k}+\frac12\right)\\ \nonumber
&=&\widetilde m+ \widehat m+\widecheck m+m'+\frac{\gamma+\delta}{2}.
\end{eqnarray}
Depending on the parity of $\gamma+\delta$, either $\rho(A)$ or $\rho(A)-\frac12$ is an integer.

Let us prove the second part of the statement. From Lemmas~\ref{suma_directa_rightsquigarrows_lem} and~\ref{type0_lem}
$$\begin{array}{lll}
\displaystyle\bigoplus_{h=1}^{\alpha} J_{4 \widetilde m_h-1} (0)
 {\rightsquigarrow}  \displaystyle
\bigoplus_{h=1}^{\alpha}H_2(-1)^{\oplus \widetilde m_h}; &&
 \displaystyle \bigoplus_{i=1}^{\beta} J_{4  \widehat m_i} (0)
\rightsquigarrow \bigoplus_{i=1}^{\beta}H_2(-1)^{\oplus\widehat m_i};\\ 
\displaystyle \bigoplus_{j=1}^{\gamma} J_{4  \widecheck m_j +1}(0)
 {\rightsquigarrow}  \displaystyle
\bigoplus_{j=1}^{\gamma} \Big(H_2(-1)^{\oplus \widecheck m_j} \oplus J_2(0)\Big); && 
\displaystyle \bigoplus_{k=1}^{\delta} J_{4 m'_k+2}(0) 
 {\rightsquigarrow}  \displaystyle
\bigoplus_{k=1}^{\delta} \Big(H_2(-1)^{\oplus m'_k} \oplus J_2(0)\Big).
\end{array}$$
Then 
$$
A{\rightsquigarrow}
H_2(-1)^{\oplus \widetilde m}\oplus H_2(-1)^{\oplus  \widehat m} \oplus \bigoplus_{j=1}^{\gamma} \Big(H_2(-1)^{\oplus \widecheck m_j} \oplus J_2(0)\Big)  \oplus \bigoplus_{k=1}^{\delta} \Big(H_2(-1)^{\oplus m'_k} \oplus J_2(0)\Big)=H_2(-1)^{\oplus \widetilde m+  \widehat m+ \widecheck m+ m'} \oplus J_2(0)^{\oplus \gamma+\delta }.
$$
With the help of Lemma~\ref{Js_lem}~(\ref{J2^2toH2(-1)}), together with (\ref{rho(A)}), we analyze the two possible cases that arise:
\begin{enumerate}
    \item If $\gamma+\delta$ is even then $\rho(A)$ is an integer and
    $$
     A \rightsquigarrow H_2(-1)^{\oplus \widetilde m+  \widehat m+ \widecheck m+ m'} \oplus J_2(0)^{\oplus \gamma+\delta } \rightsquigarrow H_2(-1)^{\oplus \widetilde m+  \widehat m+ \widecheck m+ m'+\frac{\gamma+\delta}{2}}=H_2(-1)^{\oplus \rho(A)}.
    $$
\item If $\gamma+\delta$ is odd then $\rho(A)-\frac12$ is an integer and
    $$
     A \rightsquigarrow H_2(-1)^{\oplus \widetilde m+  \widehat m+ \widecheck m+ m'} \oplus J_2(0)^{\oplus \gamma+\delta } \rightsquigarrow H_2(-1)^{\oplus \widetilde m+  \widehat m+ \widecheck m+ m'+\frac{\gamma+\delta-1}{2}}\oplus J_2(0)=H_2(-1)^{\oplus \rho(A)-\frac12}\oplus J_2(0).
    $$
\end{enumerate}
\end{proof}

\begin{theorem}\label{sufficiency-type0_th}
Let $A$ be a complex square matrix whose {\rm CFC} is a direct sum of Type-0 blocks. Then $X^\top A X= H_2(-1)^{\oplus m}$ is consistent if and only if $m\leq \rho(A)$. 
\end{theorem}
\begin{proof} 
The necessity was established  in Lemma \ref{necessary_th}.  Now, we will see that sufficiency is consequence of Theorem \ref{Js_th}. Assume   that $m\leq\rho(A)$. Since  $m\leq\rho(A)$ is equivalent to $m\leq\lfloor\rho(A)\rfloor$ then, according to the Elimination law, it suffices to prove that $A\rightsquigarrow H_2(-1)^{\oplus\lfloor \rho(A)\rfloor}$.  By Theorem \ref{Js_th} two  possibilities arise:
\begin{enumerate}
    \item  $\rho(A)$ is an integer. Then $\lfloor \rho(A)\rfloor=\rho(A)$ and the result is immediate from part \ref{rhointeger} in Theorem \ref{Js_th}.
    \item $\rho(A)-1/2$ is an integer. Then $\lfloor \rho(A)\rfloor=\rho(A)-\frac12$ and the result follows from  
    $$
    A\rightsquigarrow H_2(-1)^{\oplus \rho(A)-\frac{1}{2}}\oplus J_2(0)\rightsquigarrow H_2(-1)^{\oplus \rho(A)-\frac{1}{2}},
    $$
    where the first step is given in part \ref{rho-1/2integer} of Theorem \ref{Js_th}, and the second  is  consequence of the Elimination law.
\end{enumerate}
\end{proof}

\section{The case where CFC($A$) is a direct sum of Type-II blocks}\label{typeII_sec}

As in Section \ref{JntoB_sec}, we first provide, in Lemmas \ref{H2k+4toH2k_lem} and \ref{aux_typeII_lemma}, a collection of  auxiliary statements on the consistency. 

\begin{lemma}\label{H2k+4toH2k_lem}
For any complex numbers $\mu$ and $\nu$, the following statements hold.

\begin{enumerate}[{\rm(i)}] 
\item \label{H2k+4toH2+H2k}
$H_{2k+4}(\mu)\overset{X_0}{\rightsquigarrow} H_2(-1)\oplus H_{2k}(\mu)$, for $k\geq 0$, and $X_0={\scriptsize\left[\begin{array}{c|cc|c}
 1 &&& \\  
 -1-\mu &&&    \\ \hline
  &   & I_k & \\
  & 1  &  & \\ \hline
 &  &  & 0_{1\times k} \\
 & & & I_{k} \\
\end{array}\right]_{(2k+4)\times(2k+2)}}$.

\item \label{H2mu+H2nuToH2(-1)}
$H_2(\mu)\oplus H_2(\nu) \overset{X_0}{\rightsquigarrow} H_2(-1)$,  for $\mu\neq \nu$, with $\mu, \nu\neq -1$, and 
$X_0=
\left[\begin{smallmatrix}
 1 & 0 \\
 0 & \frac{\nu +1}{\nu -\mu } \\
 1 & 0 \\
 0 & \frac{\mu +1}{\mu -\nu }
\end{smallmatrix}\right].$

\item \label{H2mu+H2muToH2(-1)}
$H_2(\mu)^{\oplus 2}  \overset{X_0}{\rightsquigarrow}  H_2(-1)$,  for $\mu \neq \pm1$, and
$X_0=
\left[\begin{smallmatrix}
 0 & 1 \\
 \frac{1}{\mu-1} & 0 \\
 1 & 0 \\
 0 & \frac{1}{1-\mu} \\
\end{smallmatrix}\right]$.

\item \label{H2mu+J2toH2(-1)} 
$H_2(\mu)\oplus J_2(0) \overset{X_0}{\rightsquigarrow} H_2(-1)$, for $\mu\neq -1$, and 
$X_0=
\left[\begin{smallmatrix}
 1 & 0 \\
 0 & -\frac{1}{\mu } \\
 1 & 0 \\
 0 & \frac{\mu +1}{\mu}
\end{smallmatrix}\right].$
\end{enumerate}
\end{lemma}
\begin{proof} Claims~(\ref{H2mu+H2nuToH2(-1)}), (\ref{H2mu+H2muToH2(-1)}), and~(\ref{H2mu+J2toH2(-1)}) are straightforward. 
In order to prove claim~(\ref{H2k+4toH2+H2k}), let 
$$Y={\scriptsize\left[\begin{array}{cc|c} 1 & 0 & \\ -1-\mu &1 & \\ \hline & &   I_{2k+2} 
\end{array}\right]}
$$ be the  matrix of size $(2k+4)\times(2k+4)$ that differs from $I_{2k+4}$ only in the $(2,1)$ entry, which is equal to $-1-\mu$. Let also
$$
Z= 
\begin{bmatrix}
e_1 & e_3 & e_4 & \cdots & e_{k+2}& e_{k+3} & e_{k+5} & \cdots & e_{2k+4}\end{bmatrix}
$$ 
be the  matrix  obtained from $I_{2k+4}$ after deleting the second and  $(k+4)$-th columns;
and let 
$${\scriptsize P=\left[\begin{array}{c|cc|c}
1 & & &    \\ \hline
 & & I_k &    \\ 
 &1 & &     \\    \hline
 & &  & I_k
\end{array}\right]}%=\begin{bmatrix} e_2&\cdots&e_{2k+1}&e_1&-e_{2k+2} \end{bmatrix}.
$$ be the permutation matrix of size $(2k+2)\times(2k+2)$ corresponding to the cyclic permutation $(2\,3\cdots\, k+2)$. Then
$$ H_{2k+4}(\mu) \overset{YZP}{\rightsquigarrow}H_{2}(-1) \oplus H_{2k}(\mu) $$ since 
$$
%P^\top Z^\top Y^\top H_{2k+4}(\mu) Y Z  P &=
P^\top Z^\top Y^\top {\scriptsize\left[\begin{array}{cccc|cccc}
&&& &1 & && \\
&&&&&1&& \\
&&&&&&\ddots & \\
&&&&&&& 1\\ \hline
\mu&1&&&&&& \\
 & \mu &\ddots &&&&& \\
&&\ddots &1 &&&& \\
&&&\mu &&& 
\end{array}\right]} Y Z P  
=
P^\top Z^\top {\scriptsize\left[\begin{array}{cccc|cccc}
&&& &1 &-1-\mu && \\
&&&&&1&& \\
&&&&&&\ddots & \\
&&&&&&& 1\\ \hline
-1&1&&&&&& \\
-\mu-\mu^2 & \mu &\ddots &&&&& \\
&&\ddots &1 &&&& \\
&&&\mu &&& 
\end{array}\right]} Z P
$$ $$ =P^\top {\scriptsize\left[\begin{array}{cc|cc}
&&1&\\
&&&I_k\\ \hline
-1&&&\\
&J_{k}(\mu)&&\\
\end{array}\right]} P = 
{\scriptsize\left[\begin{array}{cc|cc}
 & 1 & &  \\ -1 & & & \\ \hline 
& & & I_k\\ 
& & J_{k}(\mu) & 
\end{array}\right]}=
H_2(-1) \oplus H_{2k}(\mu) .
$$
Finally, it is straightforward to check that $X_0=Y Z P$, with $X_0$ as in the statement.
\end{proof}

\begin{lemma}\label{aux_typeII_lemma} 
For $m\geq 1$, the following consistency statements hold:
\begin{enumerate}[{\rm(a)}]
    \item $H_{4m}(\mu) \rightsquigarrow H_2(-1)^{\oplus m}$ and $\rho(H_{4m}(\mu))=m$, for  $\mu\neq 0,-1$; \label{typeII_4m}
    \item $H_{4m+2}(\mu) \rightsquigarrow  H_2(-1)^{\oplus m} \oplus H_2(\mu)$ and $\rho(H_{4m+2}(\mu))=m+\frac12$, for  $\mu\neq 0,\pm 1$; \label{typeII_4m+2_mu}
    \item $H_{4m-2}(-1) \rightsquigarrow  H_2(-1)^{\oplus m}$ and $\rho(H_{4m-2}(-1))=m$. \label{typeII_4m+2_-1}
\end{enumerate}
\end{lemma}

\begin{proof} The second claim in each item is given in  Remark~\ref{rho_values}. For the first claims we will use Lemma~\ref{H2k+4toH2k_lem} (\ref{H2k+4toH2+H2k}) repeatedly:
\begin{enumerate}[(a)]
\item $ H_{4m}(\mu) \rightsquigarrow H_{2}(-1)\oplus H_{4m-4}(\mu) \rightsquigarrow \cdots  \rightsquigarrow H_2(-1)^{\oplus m}. $
\item $H_{4m+2}(\mu) \rightsquigarrow H_{2}(-1)\oplus H_{4m-2}(\mu) \rightsquigarrow \cdots \rightsquigarrow H_{2}(-1)^{\oplus m}\oplus H_{2}(\mu).$
\item For $m=1$ is trivial. If $m \geq  2$ then $H_{4m-2}(-1) \rightsquigarrow H_{2}(-1)\oplus H_{4m-6}(-1) \rightsquigarrow \cdots \rightsquigarrow H_{2}(-1)^{\oplus m}. $
\end{enumerate}
\end{proof}

Now we are ready to prove the converse of Lemma~\ref{necessary_th} when only Type-II blocks are involved in the CFC of $A$. In other words, if $A$ is a complex square matrix  whose {\rm CFC} has only Type-II blocks and  $m \leq  \rho(A)$  then  $A \rightsquigarrow H_2(-1)^{\oplus m}$. This is a direct consequence of Lemma~\ref{directsum.lemma} and the following result.

\begin{theorem}\label{Hs.th}
For a complex square matrix $A$ whose {\rm CFC} has only Type-II blocks, one of the following is satisfied:   
\begin{enumerate}[{\rm 1.}]
\item $\rho(A)$ is an integer and  $ A\rightsquigarrow H_2(-1)^{\oplus \rho(A)}$.
\item  $\rho(A)-\frac{1}{2}$ is an integer and   
$A \rightsquigarrow H_2(-1)^{\oplus  \rho(A)-\frac12} \oplus H_2(\mu)$, for some $\mu\neq 0,\pm 1$. 
\end{enumerate}
\end{theorem}

\begin{proof} The proof follows the same steps as the one of Theorem \ref{Js_th}. First, we may assume  that $A$ is in CFC. In the conditions of the statement, we can write $A$ as 
\begin{eqnarray*} 
A= 
\left(\bigoplus_{h=1}^{\alpha}H_{4\widetilde m_h}(\widetilde\mu_h)\right)
\oplus  \left(\bigoplus_{i=1}^{\beta} H_{4  \widehat m_i+2} (\widehat\mu_i)\right)
\oplus \left(\bigoplus_{j=1}^{\gamma} H_{4  \widecheck m_j -2}(-1) \right),
\end{eqnarray*}
with $\alpha, \beta,\gamma \geq 0$; $\widetilde m_1,\ldots,\widetilde m_\alpha\geq 1$; $\widetilde \mu_1,\ldots,\widetilde \mu_\alpha \neq 0,-1$; $\widehat m_1,\ldots,\widehat m_\beta \geq 0$; $\widehat\mu_1,\ldots,\widehat\mu_\beta \neq 0, \pm 1$; and
with $\widecheck m_1,\ldots,\widecheck m_\gamma \geq 1$.
Let us also define  
$$\widetilde m= \widetilde m_1+\cdots+ \widetilde m_\alpha; \quad 
\widehat m= \widehat m_1+\cdots+ \widehat m_\beta; \quad 
\widecheck m= \widecheck m_1+\cdots+ \widecheck m_\gamma.$$

From the additive property of $\rho$ (namely, \eqref{additive}) and Lemma~\ref{aux_typeII_lemma} it follows that 
\begin{eqnarray} \label{rho(A)_II}
\rho(A) &=&  
\sum_{h=1}^{\alpha}\rho(H_{4 \widetilde m_h} (\widetilde \mu_h))+
\sum_{i=1}^{\beta}\rho(H_{4 \widehat m_i+2} (\widehat \mu_i))+
\sum_{j=1}^{\gamma}\rho(H_{4 \widecheck m_j-2} (-1)) \\ \nonumber
&=&
\sum_{h=1}^{\alpha}\widetilde m_h+
\sum_{i=1}^{\beta}(\widehat m_i+\frac12)+
\sum_{j=1}^{\gamma}\widecheck m_j 
=\widetilde m+ \widehat m+\widecheck m+\frac{\beta}{2}.
\end{eqnarray}
Depending on the parity of $\beta$,  either $\rho(A)$ or $\rho(A)-\frac12$ is an integer.

Let us prove the second part of the statement. From Lemmas~\ref{suma_directa_rightsquigarrows_lem} and~\ref{aux_typeII_lemma}
$$\begin{array}{lll}
\displaystyle \bigoplus_{h=1}^{\alpha}H_{4\widetilde m_h}(\widetilde\mu_h) {\rightsquigarrow}  
\bigoplus_{h=1}^{\alpha}H_2(-1)^{\oplus \widetilde m_h}; &
 \displaystyle \bigoplus_{i=1}^{\beta} H_{4  \widehat m_i+2} (\widehat\mu_i) 
\rightsquigarrow \bigoplus_{i=1}^{\beta} \Big(H_2(-1)^{\oplus\widehat m_i}\oplus H_2(\widehat\mu_i)\Big); \\ 
\displaystyle \bigoplus_{j=1}^{\gamma} H_{4  \widecheck m_j -2}(-1)
 {\rightsquigarrow}  \displaystyle
\bigoplus_{j=1}^{\gamma} H_2(-1)^{\oplus \widecheck m_j} .
\end{array}$$
Then 
$$
A{\rightsquigarrow}
H_2(-1)^{\oplus \widetilde m}\oplus \bigoplus_{j=1}^{\beta} \Big(H_2(-1)^{\oplus \widehat m_i} \oplus H_2(\widehat \mu_i)\Big)  \oplus \bigoplus_{j=1}^{\gamma} H_2(-1)^{\oplus \widecheck m_j}=H_2(-1)^{\oplus \widetilde m+  \widehat m+ \widecheck m} \oplus \bigoplus_{j=1}^{\beta}  H_2(\widehat \mu_i).
$$
There are two possible situations:
\begin{enumerate}
    \item If $\beta$ is even, then $\rho(A)$ is an integer and
    $$
     A \rightsquigarrow H_2(-1)^{\oplus \widetilde m+  \widehat m+ \widecheck m} \oplus \bigoplus_{j=1}^{\beta}  H_2(\widehat \mu_i) \rightsquigarrow H_2(-1)^{\oplus \widetilde m+  \widehat m+ \widecheck m+\frac{\beta}{2}}=H_2(-1)^{\oplus \rho(A)}.
    $$
\item If $\beta$ is odd, then $\rho(A)-\frac12$ is an integer and
    $$
     A \rightsquigarrow H_2(-1)^{\oplus \widetilde m+  \widehat m+ \widecheck m} \oplus \bigoplus_{j=1}^{\beta}  H_2(\widehat \mu_i) \rightsquigarrow H_2(-1)^{\oplus \widetilde m+  \widehat m+ \widecheck m+\frac{\beta-1}{2}}\oplus H_2(\widehat\mu_\beta)=H_2(-1)^{\oplus \rho(A)-\frac12}\oplus H_2(\widehat\mu_\beta).
    $$
\end{enumerate}
In both cases, we have used Lemma~\ref{H2k+4toH2k_lem}~(\ref{H2mu+H2nuToH2(-1)})-(\ref{H2mu+H2muToH2(-1)}) and (\ref{rho(A)_II}).
\end{proof}

\begin{theorem}\label{sufficient-typeII.th}
Let $A$ be a complex square matrix whose {\rm CFC} is a direct sum of Type-II blocks. Then $X^\top A X = H_2(-1)^m$ is consistent if and only if $m\leq \rho(A)$. 
\end{theorem}

\begin{proof}
The necessity was established  in Lemma \ref{necessary_th}. The sufficiency is  consequence of Theorem~\ref{Hs.th}, with the argumentation  being analogous to the one given in the proof of  Theorem~\ref{sufficiency-type0_th}.
\end{proof}

\section{The case where CFC($A$) is a direct sum of Type-I blocks}\label{typeI_sec}

As in the precedent two sections, we first provide a collection of  auxiliary statements on the consistency.

\begin{lemma}
\label{Gamma2n+4->Gamma2n+H2(-1)}
\label{Gamman+1ton_lem} 
\begin{enumerate}[{\rm(i)}]

    \item \label{G2k+1toG2k}
    $\Gamma_{2k+1} \overset{X_0}{\rightsquigarrow} \Gamma_{2k}$ for $k\geq 1$ and 
    $X_0=\left[\begin{array}{cc}
        \begin{matrix}
        0 & \cdots & 0 \\ \hline
        & & \mathfrak i \vspace{-2mm}\\  & \hspace{-1mm}\iddots\hspace{-1mm} & \vspace{-2mm}\\ \mathfrak i & & 
        \end{matrix}
    \end{array} \right]_{(2k+1)\times (2k)}=\mathfrak i \begin{bmatrix}
     e_{2k+1} & e_{2k} & \cdots & e_3 & e_2
    \end{bmatrix}$.
    
\item   \label{G2k+4toG2k+H2(-1)}
$ \Gamma_{2k+4}\overset{X_0}{\rightsquigarrow}  H_2(-1)\oplus \Gamma_{2k}$, for $k\geq 1$ and $X_0={\scriptsize\left[\begin{array}{cc|c}
-1 & 0 &   \\
0 & 0  &  \\ \hline
 &  & I_{2k} \\ \hline
0 & 0 &   \\
0 & 1  &  \\
\end{array}\right]_{(2k+4)\times(2k+2)}}
$.

\item \label{G4toH2(-1)}
$\Gamma_{4} \overset{X_0}{\rightsquigarrow} H_2(-1)\oplus \Gamma_1$, for $X_0=\begin{smat}
\mathfrak i & 0 & 0  \\
0 & 0  & 0 \\
0 &  0 & \mathfrak i \\
0 & \mathfrak i & 0\\
 \end{smat}$.

\end{enumerate}
\end{lemma}
\begin{proof} Claims~(\ref{G2k+1toG2k}) and~(\ref{G4toH2(-1)}) are straightforward. In order to prove claim~(\ref{G2k+4toG2k+H2(-1)}), consider the matrix
$$
Z= 
\begin{bmatrix}
e_1 & e_3 & e_4 & \cdots & e_{2k+1}& e_{2k+2} & e_{2k+4}\end{bmatrix},
$$ 
obtained from $I_{2k+4}$ after deleting the second and the second to last column;
and the  $(2k+2)\times(2k+2)$ matrix  
$${\scriptsize P=\left[\begin{array}{c|cc}
-1 & &  \\ \hline
  & & I_{2k}   \\ 
  &1 &   \\  
\end{array}\right]}.%=\begin{bmatrix} e_2&\cdots&e_{2k+1}&e_1&-e_{2k+2} \end{bmatrix}.
$$ Then
$ \Gamma_{2k+4} \overset{ZP}{\rightsquigarrow} H_2(-1)\oplus \Gamma_{2k} $, since 
$$
P^\top Z^\top \Gamma_{2k+4}  Z  P =
P^\top \left[\begin{array}{c|c|c}
 & & -1 \\ \hline
 & \Gamma_{2k} & \\ \hline
1 & &  
\end{array}\right] P =
H_2(-1) \oplus \Gamma_{2k}.
$$
Finally, it is straightforward to check that $X_0=Z P$, with $X_0$ as in the statement.
\end{proof}

\subsection
{A direct sum of $\Gamma_2$ blocks }\label{gamma_2.sec} 

The  main obstacle for the necessary condition of Lemma~\ref{necessary_th} to be also sufficient is the presence in CFC($A$) of either blocks $\Gamma_1$ or $\Gamma_2$. Regarding $\Gamma_2$, we will see in Theorem~\ref{gamma2toh2(-1).th} that, if the  CFC of $A$ is $\Gamma_2^{\oplus k}$, with $k\geq 3$, then \break{$\Gamma_2^{\oplus k} \rightsquigarrow\hspace{-3ex}/ \quad H_2(-1)^{\oplus \left\lfloor\rho(A)\right\rfloor}$}.

To prove Theorem~\ref{gamma2toh2(-1).th} we will use the following three  results. The first one is well-known, and we omit the proof.

\begin{lemma}\label{sym-skew.lem}
Let $A+B$ be skew-symmetric, with $A$ being symmetric and $B$ being skew-symmetric. Then $A=0$.
\end{lemma}
%\begin{proof}
%By hypothesis $(A+B)^\top=-(A+B)$, which implies $A^\top+B^\top=-A-B$. But, since $A^\top=A$ and $B^\top=-B$, this is equivalent to $A-B=-A-B$, which implies $A=0$, as claimed.
%\end{proof}

\begin{lemma}\label{yrank.lem}
Let $Y\in\CC^{r\times s}$ be such that $\rank Y>\frac{s}{2}$. Then $YY^\top\neq0$.
\end{lemma}
\begin{proof}
It is a direct consequence of the identity $\rank Y^\top= \rank Y >\frac{s}{2}$ and the fact that $\dim(\ker Y)<\frac{s}{2}$.
\end{proof}

\begin{lemma}\label{rankybound.lem} Let $X^\top$ have $2k$ columns, denoted by $x_i$, for $i=1,\hdots,2k$. Then
$$
\rank \left(X^\top H_2(-1)^{\oplus k} X \right)\leq 2\cdot\rank \begin{bmatrix} x_2&x_4&\hdots&x_{2k} \end{bmatrix}.
$$
\end{lemma}

\begin{proof}
If $\rank \begin{bmatrix} x_2&x_4&\hdots&x_{2k} \end{bmatrix}=k$ then the inequality is clearly satisfied.
Without loss of generality we may assume that $\rank \begin{bmatrix} x_2&x_4&\hdots&x_{2k} \end{bmatrix}=s$, with $x_2,x_4,\ldots,x_{2s}$  linearly independent vectors. 
Then $$
x_{2s+2}=\sum_{i=1}^{s}\alpha^{(2s+2)}_{2i}x_{2i},\quad\ldots,\quad
x_{2k}=\sum_{i=1}^{s}\alpha^{(2k)}_{2i}x_{2i}\,, \quad \text{ for some } \alpha^{(2s+2)}_{2},\hdots,\alpha^{(2s+2)}_{2s},\hdots,\alpha_{2}^{(2k)},\hdots,\alpha_{2s}^{(2k)}\in\mathbb C.
$$ 
Let $\mathbb{A}$ be represent the bilinear form defined by
$\mathbb{A}(u, v)=u v^\top-v u^\top$ for $u, v\in \mathbb{C}$. Note that $\rank \mathbb{A}(u, v)\leq 2$. Then
\begin{align*}
& X^\top (H_2(-1)^{\oplus k})X = \mathbb{A}(x_1, x_2)+\mathbb{A}(x_3, x_4)+\cdots+\mathbb{A}(x_{2k-1}, x_{2k})\\ 
&\qquad = \mathbb{A}(x_1, x_2)+\mathbb{A}(x_3, x_4)+\cdots+\mathbb{A}(x_{2s-1}, x_{2s})+\mathbb{A}\left(x_{2s+1}, \sum_{i=1}^{s}\alpha^{(2s+2)}_{2i}x_{2i}\right)
+\cdots+\mathbb{A}\left(x_{2k-1}, \sum_{i=1}^{s}\alpha^{(2k)}_{2i}x_{2i}\right)\\
&\qquad = \mathbb{A}(x_1, x_2)+\mathbb{A}(x_3, x_4)+\cdots+\mathbb{A}(x_{2s-1}, x_{2s})+\sum_{i=1}^{s}\alpha^{(2s+2)}_{2i}\mathbb{A}\left(x_{2s+1}, x_{2i}\right)
+\cdots+\sum_{i=1}^{s}\alpha^{(2k)}_{2i}\mathbb{A}\left(x_{2k-1}, x_{2i}\right)\\
&\qquad =\mathbb{A}\left(x_1+\alpha^{(2s+2)}_{2} x_{2s+1}+\cdots+\alpha^{(2k)}_{2} x_{2k-1},\; x_2\right)+\cdots+\mathbb{A}\left(x_{2s-1}+\alpha^{(2s+2)}_{2s} x_{2s+1}+\cdots+\alpha^{(2k)}_{2s} x_{2k-1},\; x_{2s}\right).
\end{align*}
Since this is a sum of $s$ addends, each one having rank at most two, we get the desired inequality.
\end{proof}

Now we will show that, for $A=\Gamma_2^{\oplus k}$, the converse of Lemma~\ref{necessary_th} is not true. Moreover, we are able to characterize the consistency of Eq.~\eqref{maineq} with $A=\Gamma_2^{\oplus k}$ and $B=H_2(-1)^{\oplus m}$.

\begin{theorem}\label{gamma2toh2(-1).th}
$\Gamma_2^{\oplus k} {\rightsquigarrow} H_2(-1)^{\oplus m}$ if and only if $m\leq\displaystyle \frac{2}{3}\,\rho(\Gamma_2^{\oplus k})=\frac{k}{2}$.
\end{theorem}
\begin{proof}
We first prove the necessity. Assume that there is some $X\in\CC^{(2k)\times(2m)}$ such that 
\begin{equation}\label{3gamma2to2h2(-1)}
X^\top \Gamma_2^{\oplus k} X = H_2(-1)^{\oplus m}.
\end{equation}
As $\Gamma_2=-H_2(-1)+\begin{smat}
 0 & 0 \\ 0 & 1 \end{smat}$ then 
\begin{equation}\label{xsum}
X^\top \Gamma_2^{\oplus k}X
=-X^\top H_2(-1)^{\oplus k}X+X^\top \begin{bmatrix}
 0&0\\0&1
\end{bmatrix}^{\oplus k}X,
\end{equation}
with the first addend in the right-hand side of \eqref{xsum} being skew-symmetric, and the second one being symmetric. As a consequence of Lemma \ref{sym-skew.lem} and \eqref{3gamma2to2h2(-1)}, we conclude that the second addend in the right-hand side of \eqref{xsum} is equal to zero. Note that this addend is equal to $X^\top_{\rm even}X_{\rm even}$, where $X^\top_{\rm even}=\begin{bmatrix}
x_2&x_4&\hdots&x_{2k} 
\end{bmatrix}$, with $x_i$ being the $i$th column of $X^\top$ (namely, the $i$th row of $X$). Then, Lemma \ref{yrank.lem} implies that $\rank (X^\top_{\rm even})\leq \frac{k}{2}$, and Lemma \ref{rankybound.lem} in turn implies that 
$$\rank (X^\top H_2(-1)^{\oplus k}X)\leq 2\cdot \rank X^\top_{\rm even} \leq 2\cdot \frac{k}{2}=k.$$
Since 
$$2m=\rank(H_2(-1)^{\oplus m})= \rank (X^\top \Gamma_2^{\oplus k}X)= \rank (X^\top H_2(-1)^{\oplus k}X)\leq k =\frac{4\rho(\Gamma_2^{\oplus k})}{3}, $$ 
the necessity  follows.

For the sufficiency, it is enough to prove, as a consequence of Lemma \ref{directsum.lemma}, that 
$$\Gamma_2^{\oplus k} {\rightsquigarrow} H_2(-1)^{\oplus m}  \quad \text{ for } m=  \left\lfloor\frac{2}{3}\rho(\Gamma_2^{\oplus k}) \right\rfloor.$$
First note that $\Gamma_2^{\oplus 2} \overset{X_0}{\rightsquigarrow} H_2(-1)$, with  $X_0=\begin{smat} 0&-1\\ 1&0\\ 0 & 0 \\ \mathfrak i &0 \end{smat}$, and so:
\begin{enumerate}[(i)]
    \item If $k=2t$ then $\lfloor\frac{2}{3}\rho(\Gamma_2^{\oplus 2t}) \rfloor= \lfloor\frac{2}{3}\cdot \frac{6t}{4} \rfloor=t$ and $\Gamma_2^{\oplus 2t} \overset{X_1}{\rightsquigarrow} H_2(-1)^{\oplus t}$, with $X_1=X_0^{\oplus t}$.
    \item If $k=2t+1$ then  $\lfloor\frac{2}{3}\rho(\Gamma_2^{\oplus 2t+1}) \rfloor=\lfloor\frac{2}{3}\cdot \frac{6t+3}{4} \rfloor=\lfloor t+\frac{1}{2} \rfloor=t$  and $\Gamma_2^{\oplus 2t+1} \rightsquigarrow \Gamma_2^{\oplus 2t} \rightsquigarrow  H_2(-1)^{\oplus t}$.
\end{enumerate}

\end{proof}

\subsection{Why the blocks $\Gamma_1$ and $\Gamma_2$ are problematic}\label{type2blocksbad.sec}

Theorem \ref{gamma2toh2(-1).th} shows that the necessary condition provided in Lemma \ref{necessary_th} for Eq. \eqref{maineq} to be consistent is not sufficient for an arbitrary $A$. In particular, for a matrix $A$ whose CFC contains Type-I blocks of size $2$, namely $\Gamma_2$. For instance, $\Gamma_2^{\oplus4}\rightsquigarrow\hspace{-3ex}/ \quad H_2(-1)^{\oplus3}$, but the necessary condition $3\leq\rho(\Gamma_2^{\oplus4})=3$ is satisfied. 

Something similar occurs with Type-I blocks of size $1$, namely $\Gamma_1=[1]$. For instance, $\Gamma_1^{\oplus n} \rightsquigarrow\hspace{-3ex}/ \quad H_2(-1)$ for any $n$, because $\Gamma_1^{\oplus n}$ is symmetric and $H_2(-1)$ is skew-symmetric. However, the necessary condition $1\leq\rho(\Gamma_1^{\oplus n})=\frac{n}{2}$ is satisfied for $n\geq2$.

The remaining of this section is devoted to prove, in Theorem \ref{sufficiency-type1.th}, that, if the CFC of $A$ does not contain either blocks $\Gamma_1$ or $\Gamma_2$, then the necessary condition of Lemma \ref{necessary_th} is also sufficient. As a consequence of the previous arguments, the case where the CFC of $A$ contains blocks of either type $\Gamma_1$ or $\Gamma_2$ requires a different analysis, that is not addressed in this work.

\subsection{The case of Type-I blocks of size at least 3}
The goal of this section is to prove Theorem \ref{sufficiency-type1.th}, which is the counterpart of Theorems \ref{sufficiency-type0_th} and \ref{sufficient-typeII.th} when the CFC of $A$ only contains Type-I blocks (except $\Gamma_1$ and $\Gamma_2$). In other words, that the necessary condition stated in Lemma \ref{necessary_th} is also sufficient in this case. For this, we need a series of technical results, namely Lemmas \ref{Basics_for_Gamma}--\ref{TypeIBlocks_size_geq3_th}, that we state and prove all in a row.

The following result provides some basic reductions involving Type-I blocks.

\begin{lemma} \label{basics_G3_G6} \label{Basics_for_Gamma}
\begin{enumerate}[{\rm(i)}]

\item \label{J2+G1toG2}
$J_2(0) \oplus \Gamma_1   \overset{X_0}{\rightsquigarrow} \Gamma_2$, with $X_0=\begin{smat} -1 & 1 \\ 1 & 1 \\ 1  & 0 \end{smat}$.

\item \label{G2+G1toH2}
$\Gamma_2 \oplus \Gamma_1   \overset{X_0}{\rightsquigarrow} H_2(-1)$, with $X_0=\begin{smat} 0 & 1 \\ 1 & 0 \\ \mathfrak{i}  & 0 \end{smat}$.

\item \label{J2+G1^2toH2(-1)}
 $J_{2}(0)\oplus {\Gamma_{1}}^{\oplus 2} \overset{X_0}{\rightsquigarrow} H_2(-1)$,  with
$X_0=\sqrt{2}\begin{smat}
 1 & 0 \\
 0 & 1 \\
 \frac{\mathfrak i}{2} & \frac{\mathfrak i}{2} \\
 \frac{1}{2} & -\frac{1}{2} \\
\end{smat}.$

\item \label{G2+J2toH2+G1} 
$\Gamma_2 \oplus J_2(0)   \overset{X_0}{\rightsquigarrow} H_2(-1)\oplus \Gamma_1$,  with $X_0=\begin{smat} -1 & 0 & -1 \\ 0 & 1 & 0 \\ 0 & 1 & 1 \\ 0 & -1 & 1 \end{smat}$.
    
\item  \label{G3+G2toH2+J2}        
    $\Gamma_3 \oplus \Gamma_2 \overset{X_0}{\rightsquigarrow} H_2(-1)\oplus J_2(0)$, with  $X_0=\begin{smat}
         0 & 0 & 0 & 1 \\
 1 & 0 & 0 & 1 \\
 1 & 0 & \frac{1}{2} & \frac{1}{2} \\
 0 & 1 & \frac{1}{2} & -\frac{1}{2} \\
 1 & 0 & 0 & 0 \\
         \end{smat}$.

\item \label{G6+G2toH2^2+J2}   
$\Gamma_6\oplus \Gamma_2 \overset{X_0}{\rightsquigarrow} H_2(-1)^{\oplus 2}\oplus J_2(0)$
         for $X_0=
\begin{smat}
 -1 & 0 & 0 & 0 & 0 & 0 \\
 0 & 0 & 0 & 0 & 0 & 0 \\
 0 & 0 & 0 & -\mathfrak{i} & -\frac{1}{2} & \frac{1}{2} \\
 0 & 0 & 0 & -\mathfrak{i} & 0 & 1 \\
 0 & 0 & 0 & 0 & 0 & 1 \\
 0 & 1 & 0 & 0 & 0 & 0 \\
 0 & 0 & -1 & 0 & -\frac{\mathfrak{i}}{2} & -\frac{\mathfrak{i}}{2} \\
 0 & 0 & 0 & 1 & 0 & 0 \\
\end{smat}$.

\item  \label{G2^4toH2(-1)2+J2}    $\Gamma_4^{\oplus 2} \overset{X_0}{\rightsquigarrow} H_2(-1)^{\oplus 2}\oplus J_2(0)$, with 
$X_0 =  
\begin{smat}
 -1 & 0 & 0 & 0 & 0 & 0 \\
 0 & 0 & 0 & 0 & 0 & 0 \\
 0 & 0 & 0 & \mathfrak{i} & -\frac{\mathfrak{i}}{2} & -\frac{\mathfrak{i}}{2} \\
 0 & 1 & 0 & 0 & 0 & 0 \\
 0 & 0 & -1 & 0 & -\frac{1}{4} & \frac{1}{4} \\
 0 & 0 & 0 & -\frac{1}{2} & \frac{1}{2} & \frac{1}{2} \\
 0 & 0 & 0 & 0 & -\frac{1}{2} & \frac{1}{2} \\
 0 & 0 & 0 & 1 & 0 & 0 \\
\end{smat}$.
\end{enumerate}
\end{lemma}

The proof of Lemma \ref{Basics_for_Gamma} is straightforward, just by checking that $X_0$ indeed satisfies the corresponding equations.

In Lemmas~\ref{Gammas_th} and \ref{basic_auxiliar_Gamma4k+2_th} we consider the case where $A$ is a direct sum of Type-I blocks. In particular, Lemma~\ref{Gammas_th} addresses the sum of blocks of the form $\Gamma_{4k+\epsilon}$, for $\epsilon\in\{0,1\}$, and Lemma \ref{basic_auxiliar_Gamma4k+2_th} the sum of blocks of the form $\Gamma_{4k+\epsilon}$, with $\epsilon\in\{2,3\}$.

\begin{lemma} \label{Gammas_th}
For each $r\geq 0$, $k_i\geq 1$, and  $\epsilon_i\in\{0,1\}$, we have:
\begin{enumerate}[{\rm(i)}]
    \item \label{4rG4k}
    $ A=\bigoplus_{i=1}^{4r} \Gamma_{4k_i+\epsilon_i} \rightsquigarrow  H_2(-1)^{\oplus \rho(A)}$ with $\rho(A)=k_1+\cdots+k_{4r}+r$.
    \item \label{4r+1G4k}
    $A=\bigoplus_{i=1}^{4r+1} \Gamma_{4k_i+\epsilon_i} \rightsquigarrow  H_2(-1)^{\oplus \rho(A)-\frac14}\oplus \Gamma_1$ with $\rho(A)=k_1+\cdots+k_{4r+1}+r+\frac14$.
    \item \label{4r+2G4k}
    $ A=\bigoplus_{i=1}^{4r+2} \Gamma_{4k_i+\epsilon_i} \rightsquigarrow  H_2(-1)^{\oplus \rho(A)-\frac12}\oplus J_2(0)$ with $\rho(A)=k_1+\cdots+k_{4r+2}+r+\frac12$.
    \item \label{4r+3G4k}
    $A=\bigoplus_{i=1}^{4r+3} \Gamma_{4k_i+\epsilon_i} \rightsquigarrow  H_2(-1)^{\oplus \rho(A)-\frac34}\oplus J_2(0)\oplus \Gamma_1$ with $\rho(A)=k_1+\cdots+k_{4r+3}+r+\frac34$.
 \end{enumerate}
\end{lemma}
\begin{proof} In all items, the calculation of $\rho(A)$  is a direct task  by using Remark~\ref{rho_values}. In order to prove the first claim of the statement for each item, we start by proving it for $r=0$ in cases (ii)--(iv), and for $r=1$ in case (i) (namely, when $A$ contains exactly one, two, three, and four blocks, respectively). Then, we address the general case of an arbitrary $r$.
\begin{itemize}%[1.]
    \item   For a single block, namely item~(\ref{4r+1G4k}) when $r=0$, we have that      
    \begin{equation}\label{OneBlock}
        \Gamma_{4k+\epsilon} \rightsquigarrow  \Gamma_{4k} \rightsquigarrow   H_2(-1)^{\oplus k-1}\oplus \Gamma_{4}
        \rightsquigarrow  H_2(-1)^{\oplus k} \oplus \Gamma_{1},
    \end{equation}
    where the first step (that appears only when $\epsilon=1$)  is due to Lemma~\ref{Gamma2n+4->Gamma2n+H2(-1)}~(\ref{G2k+1toG2k}), the second one is a consequence of applying $k-1$ times  Lemma~\ref{Gamma2n+4->Gamma2n+H2(-1)}~(\ref{G2k+4toG2k+H2(-1)}),  and the third step is due to   Lemma~\ref{Gamma2n+4->Gamma2n+H2(-1)}~\eqref{G4toH2(-1)}.

    \item  For two blocks, namely item~(\ref{4r+2G4k}) when $r=0$, we perform the same transformations for two Type-I blocks as we did for only one block in the previous item, except for the last step, so that we have
      \begin{equation}\label{TwoBlocks}
      \Gamma_{4k_1+\epsilon_1} \oplus \Gamma_{4k_2+\epsilon_2}    \rightsquigarrow   (H_2(-1)^{\oplus k_1-1} \oplus \Gamma_{4}) \oplus (H_2(-1)^{\oplus k_2-1} \oplus \Gamma_{4}) \rightsquigarrow H_2(-1)^{\oplus k_1+k_2} \oplus J_2(0),\end{equation} 
     where this last step is  due to  Lemma~\ref{Basics_for_Gamma} (\ref{G2^4toH2(-1)2+J2}).

    \item The case of three blocks, namely item~(\ref{4r+3G4k}) when $r=0$, follows from \eqref{OneBlock} and \eqref{TwoBlocks}:
      \begin{equation}\label{ThreeBlocks}
      \bigoplus_{i=1}^3\Gamma_{4k_i+\epsilon_i}   \rightsquigarrow  \Big(H_2(-1)^{\oplus k_1+k_2} \oplus J_2(0)\Big)\oplus \Big(H_2(-1)^{\oplus k_3} \oplus \Gamma_{1}\Big) \rightsquigarrow H_2(-1)^{\oplus k_1+k_2+k_3} \oplus J_2(0) \oplus \Gamma_{1}.\end{equation} 

    \item As for the case of four blocks, namely item~(\ref{4rG4k}) when $r=1$, it follows from \eqref{TwoBlocks}:
      \begin{equation}\label{FourBlocks}
      \bigoplus_{i=1}^4\Gamma_{4k_i+\epsilon_i}    \rightsquigarrow  \Big(H_2(-1)^{\oplus k_1+k_2} \oplus J_2(0)\Big)\oplus \Big(H_2(-1)^{\oplus k_3+k_4} \oplus J_2(0)\Big) \rightsquigarrow H_2(-1)^{\oplus k_1+k_2+k_3+k_4+1},\end{equation} 
where the last step follows from $J_2(0)^{\oplus 2}\rightsquigarrow H_2(-1)$, as proved in Lemma~\ref{Js_lem}~(\ref{J2^2toH2(-1)}).
\end{itemize}

Now we are ready to prove the general result for any $r$:
\begin{enumerate}[(i)]
     \item The result follows from \eqref{FourBlocks} and the Addition law of consistency, by forming groups of four addends. More precisely,
    \begin{equation}\label{FourRBlocks}     
    \bigoplus_{i=1}^{4r} \Gamma_{4k_i+\epsilon_i}
     \rightsquigarrow  H_2(-1)^{\oplus k_1+\cdots+k_{4}+1}\oplus \cdots \oplus H_2(-1)^{\oplus k_{4r-3}+\cdots+k_{4r}+1} 
    =  H_2(-1)^{\oplus k_1+\cdots+k_{4r}+r}.
    \end{equation}
   \item The result follows from \eqref{FourRBlocks} and \eqref{OneBlock}, together with the Addition law.     
   \item The result follows from \eqref{FourRBlocks} and \eqref{TwoBlocks}, together with the Addition law. 
    \item Again, the result follows from \eqref{FourRBlocks} and \eqref{ThreeBlocks}, together with the Addition law. 
\end{enumerate}
\end{proof}

\begin{lemma} \label{basic_auxiliar_Gamma4k+2_th}
For each $r\geq 0$, $k_i\geq 0$, and $\epsilon_i\in\{0,1\}$ with $(k_i,\epsilon_i)\neq (0,0)$, we have:
\begin{enumerate}[{\rm(i)}]
    \item  \label{4rGammas4k+2} 
    $A=\bigoplus_{i=1}^{4r} \Gamma_{4k_i+2+\epsilon_i} \rightsquigarrow  H_2(-1)^{\oplus \rho(A)}$ with $\rho(A)=k_1+\cdots+k_{4r}+3r$.
    \item \label{4r+1Gamma4k+2}
    $  A=\bigoplus_{i=1}^{4r+1} \Gamma_{4k_i+2+\epsilon_i} \rightsquigarrow  H_2(-1)^{\oplus \rho(A)-\frac34}\oplus \Gamma_2$  with $\rho(A)=k_1+\cdots+k_{4r+1}+3r+\frac34$.
    \item \label{4r+2Gammas4k+2}
    $ A=\bigoplus_{i=1}^{4r+2} \Gamma_{4k_i+2+\epsilon_i} \rightsquigarrow  H_2(-1)^{\oplus \rho(A)-\frac12}\oplus J_2(0)$ with $\rho(A)=k_1+\cdots+k_{4r+2}+3r+\frac32$.
    \item \label{4r+3Gammas4k+2}
    $A=\bigoplus_{i=1}^{4r+3} \Gamma_{4k_i+2+\epsilon_i} \rightsquigarrow  H_2(-1)^{\oplus \rho(A)-\frac14}\oplus \Gamma_1$  with $\rho(A)=k_1+\cdots+k_{4r+3}+3r+\frac94$.
 \end{enumerate}
\end{lemma}
\begin{proof} We follow the same steps as in the proof of Lemma \ref{Gammas_th}. In all items, the calculation of the value of $\rho(A)$  is a direct task by using Remark~\ref{rho_values} and the Addition law. For instance, in part (ii), we have
$$
\rho(A)=\frac{\sum_{i=1}^{4r+1}(4k_i+2+\epsilon_i)+\sum_{i=1}^{4r+1}(1-\epsilon_i)}{4}=\sum_{i=1}^{4r+1}k_i+\frac{2(4r+1)+4r+1}{4}=\sum_{i=1}^{4r+1}k_i+3r+\frac{3}{4}.
$$
In order to prove the first part of the statement for each item, we first prove the cases $r=0$ in parts (ii)--(iv) and $r=1$ in part (i), and then we address the general case of an arbitrary $r$ in all parts.
\begin{itemize}
        \item   For a single block, namely item~(\ref{4r+1Gamma4k+2}) when $r=0$, we have that      
    \begin{equation}\label{OneBlockG4k+2}
        \Gamma_{4k+2+\epsilon} \rightsquigarrow  \Gamma_{4k+2} \rightsquigarrow   H_2(-1)^{\oplus k}\oplus \Gamma_{2},
    \end{equation}
    where the first step (arising only if $\epsilon=1$) is due to Lemma~\ref{Gamma2n+4->Gamma2n+H2(-1)}~(\ref{G2k+1toG2k}), and the second one is a consequence of applying $k$ times  Lemma~\ref{Gamma2n+4->Gamma2n+H2(-1)}~(\ref{G2k+4toG2k+H2(-1)}).
    
        \item  For two blocks, namely item~(\ref{4r+2Gammas4k+2}) when $r=0$, we are going to prove that 
    \begin{equation}\label{TwoBlocksG4k+2}
        \Gamma_{4k_1+2+\epsilon_1} \oplus \Gamma_{4k_2+2+\epsilon_2} \rightsquigarrow     H_2(-1)^{\oplus k_1+k_2+1}\oplus J_{2}(0).
    \end{equation}
        For this we consider the following three possible cases:
        \begin{enumerate}
    \item If $k_1=k_2=0$ then, by hypothesis, $\epsilon_1=\epsilon_2=1$. We have that
             $$
             \Gamma_3\oplus \Gamma_3\rightsquigarrow H_2(-1)\oplus J_2(0),
             $$
    by Lemma~\ref{Gamma2n+4->Gamma2n+H2(-1)}~(\ref{G2k+1toG2k}) and  Lemma~\ref{basics_G3_G6}~(\ref{G3+G2toH2+J2}).
    \item If $k_1\geq 1$ then, as $\Gamma_{4k+3} \rightsquigarrow \Gamma_{4k+2}$ by Lemma~\ref{Gamma2n+4->Gamma2n+H2(-1)}~(\ref{G2k+1toG2k}), we may assume that  $\epsilon_1=\epsilon_2=0$. Then:
              \begin{equation*}
              \Gamma_{4k_1+2}\oplus \Gamma_{4k_2+2} \rightsquigarrow \big(H_2(-1)^{\oplus k_1-1}\oplus \Gamma_{6}\big) \oplus \big(H_2(-1)^{\oplus k_2}\oplus \Gamma_{2}\big)  \rightsquigarrow H_2(-1)^{\oplus k_1+k_2+1}\oplus J_2(0),\end{equation*}
    where the first step is due to Lemma~\ref{Gamma2n+4->Gamma2n+H2(-1)}~(\ref{G2k+4toG2k+H2(-1)}) applied repeatedly, and the second one is due to Lemma~\ref{basics_G3_G6}~(\ref{G6+G2toH2^2+J2}).

    \item A similar argument is valid if $k_2\geq 1$. 
\end{enumerate}

          \item The case of three blocks, namely item~(\ref{4r+3Gammas4k+2}) when $r=0$,  follows from \eqref{TwoBlocksG4k+2} and \eqref{OneBlockG4k+2}:
      \begin{equation}\label{ThreeBlocksG4k+2}
      \Big(\Gamma_{4k_1+2+\epsilon_1}\oplus \Gamma_{4k_2+2+\epsilon_2}\Big)\oplus \Gamma_{4k_3+2+\epsilon_3}\rightsquigarrow H_2(-1)^{\oplus k_1+k_2+k_3+1} \oplus \Gamma_2 \oplus  J_2(0) \rightsquigarrow H_2(-1)^{\oplus k_1+k_2+k_3+2} \oplus \Gamma_1,\end{equation} 
    where the last step is due to Lemma~\ref{basics_G3_G6}~(\ref{G2+J2toH2+G1}).

    \item As for the case of four blocks, namely item~(\ref{4rG4k}) when $r=1$, it follows from \eqref{TwoBlocksG4k+2}:
      \begin{equation}\label{FourBlocksG4k+2}
      \bigoplus_{i=1}^4\Gamma_{4k_i+2+\epsilon_i}    \rightsquigarrow  \Big(H_2(-1)^{\oplus k_1+k_2+1} \oplus J_2(0)\Big)\oplus \Big(H_2(-1)^{\oplus k_3+k_4+1} \oplus J_2(0)\Big) \rightsquigarrow H_2(-1)^{\oplus k_1+k_2+k_3+k_4+3},\end{equation} 
where the last step is due to  Lemma~\ref{Js_lem}~(\ref{J2^2toH2(-1)}).  
\end{itemize}

Now we are ready to prove the general case, for an arbitrary $r$:
\begin{enumerate}[(i)]
     \item The result follows from \eqref{FourBlocksG4k+2} and the Addition law of consistency. Namely,
    \begin{equation}\label{FourRBlocksG4k+2}     
    \bigoplus_{i=1}^{4r} \Gamma_{4k_i+2+\epsilon_i}
     \rightsquigarrow  H_2(-1)^{\oplus k_1+\cdots+k_{4}+3}\oplus \cdots \oplus H_2(-1)^{\oplus k_{4r-3}+\cdots+k_{4r}+3} 
    =  H_2(-1)^{\oplus k_1+\cdots+k_{4r}+3r}.
    \end{equation}
   \item The result follows from \eqref{FourRBlocksG4k+2} and \eqref{OneBlockG4k+2}, together with the Addition law.     
   \item The result follows from \eqref{FourRBlocksG4k+2} and \eqref{TwoBlocksG4k+2}, together with the Addition law. 
    \item Again, the result follows from \eqref{FourRBlocksG4k+2} and \eqref{ThreeBlocksG4k+2}, together with the Addition law. 
\end{enumerate}
\end{proof}

From Lemmas \ref{Gammas_th} and \ref{basic_auxiliar_Gamma4k+2_th} we can obtain the following.

\begin{theorem}\label{TypeIBlocks_size_geq3_th}
For a complex square matrix, $A$, whose {\rm CFC} has only Type-I blocks of size at least $3$, one of the following is satisfied:   
\begin{enumerate}[{\rm(i)}]
\item $\rho(A)$ is an integer and  $ A\rightsquigarrow H_2(-1)^{\oplus \rho(A)}$.
\item  $\rho(A)-\frac{1}{4}$ is an integer and   
$A \rightsquigarrow H_2(-1)^{\oplus  \rho(A)-\frac14} \oplus \Gamma_1$. 
\item  $\rho(A)-\frac{1}{2}$ is an integer and   
$A \rightsquigarrow H_2(-1)^{\oplus  \rho(A)-\frac12} \oplus C$ where $C$ is either $J_2(0)$ or $\Gamma_1^{\oplus 2}$. 
\item  $\rho(A)-\frac{3}{4}$ is an integer and   
$A \rightsquigarrow H_2(-1)^{\oplus  \rho(A)-\frac34} \oplus \Gamma_2$. 
\end{enumerate}
\end{theorem}

\begin{proof}
By hypothesis, 
$$A=A'\oplus A''=\bigoplus_{i=1}^{4j+k} \Gamma_{4h_i+\epsilon_i} \oplus \bigoplus_{r=1}^{4s+t} \Gamma_{4u_{r}+2+\eta_{r}}\,, $$
for some $j,s\geq 0$, $k,t\in\{0,1,2,3\}$, $h_i\geq 1$,  $\epsilon_i\in\{0,1\}$; and  $u_{r}\geq 0$,  $\eta_{r}\in\{0,1\}$ with $(u_{r},\eta_{r})\neq (0,0)$.  

From the additive property of $\rho$, namely Eq.  \eqref{additive}, $\rho(A)=\rho(A')\oplus \rho(A'')$,
and  looking at the values of $\rho(A')$ in Lemma \ref{Gammas_th} and $\rho(A'')$ in Lemma \ref{basic_auxiliar_Gamma4k+2_th}, the following possibilities arise for $\rho(A)$. Moreover, the relations $\rightsquigarrow$ that we obtain are also consequences of Lemmas \ref{Gammas_th} and \ref{basic_auxiliar_Gamma4k+2_th}, unless otherwise stated:
\begin{enumerate}[(i)]
    \item $\rho(A)$ is an integer when either: 
    \begin{itemize}
        \item $k=t=0$, then both $\rho(A')$ and $\rho(A'')$ are integers and
        $$
        A'\oplus A'' \rightsquigarrow H_2(-1)^{\oplus\rho(A')} \oplus  H_2(-1)^{\oplus\rho(A'')} = H_2(-1)^{\oplus\rho(A)}.
        $$
        \item $k=t=1$, then both $\rho(A')-\frac14$ and $\rho(A'')-\frac34$ are integers and
        $$
        A'\oplus A'' \rightsquigarrow H_2(-1)^{\oplus\rho(A')-\frac14}\oplus \Gamma_1 \oplus  H_2(-1)^{\oplus\rho(A'')-\frac34}\oplus \Gamma_2 = H_2(-1)^{\oplus\rho(A)},
        $$
        where the last step is a consequence of Lemma \ref{basics_G3_G6} \eqref{G2+G1toH2}.
        \item $k=t=2$, then both $\rho(A')-\frac12$ and $\rho(A'')-\frac12$ are integers and
        $$
        A'\oplus A'' \rightsquigarrow H_2(-1)^{\oplus\rho(A')-\frac12}\oplus J_2(0) \oplus  H_2(-1)^{\oplus\rho(A'')-\frac12}\oplus J_2(0) \rightsquigarrow  H_2(-1)^{\oplus\rho(A)},
        $$
        where the last step is a consequence of Lemma \ref{Js_lem}~\eqref{J2^2toH2(-1)}.
        \item $k=t=3$, then both $\rho(A')-\frac34$ and $\rho(A'')-\frac14$ are integers and
        $$
        A'\oplus A'' \rightsquigarrow H_2(-1)^{\oplus\rho(A')-\frac34}\oplus J_2(0) \oplus \Gamma_1 \oplus  H_2(-1)^{\oplus\rho(A'')-\frac14}\oplus \Gamma_1 \rightsquigarrow  H_2(-1)^{\oplus\rho(A)},
        $$
        where the last step is a consequence of Lemma \ref{Basics_for_Gamma} \eqref{J2+G1^2toH2(-1)}.
    \end{itemize}    
    \item $\rho(A)-\frac14$ is an integer when either:
\begin{itemize}
    \item $k=0$ and $t=3$, then both $\rho(A')$ and $\rho(A'')-\frac14$ are integers and
        $$
        A'\oplus A'' \rightsquigarrow H_2(-1)^{\oplus\rho(A')} \oplus  H_2(-1)^{\oplus\rho(A'')-\frac14} \oplus \Gamma_1= H_2(-1)^{\oplus\rho(A)-\frac14}\oplus \Gamma_1.
        $$
    \item $k=1$ and $t=0$, then both $\rho(A')-\frac14$ and $\rho(A'')$ are integers and
        $$
        A'\oplus A'' \rightsquigarrow H_2(-1)^{\oplus\rho(A')-\frac14}  \oplus \Gamma_1 \oplus  H_2(-1)^{\oplus\rho(A'')}= H_2(-1)^{\oplus\rho(A)-\frac14}\oplus \Gamma_1.
        $$
    \item $k=2$ and $t=1$, then both $\rho(A')-\frac12$ and $\rho(A'')-\frac34$ are integers and
        $$
        A'\oplus A'' \rightsquigarrow H_2(-1)^{\oplus\rho(A')-\frac12}\oplus J_2(0) \oplus  H_2(-1)^{\oplus\rho(A'')-\frac34}\oplus \Gamma_2 \rightsquigarrow  H_2(-1)^{\oplus\rho(A)-\frac14}\oplus \Gamma_1,
        $$
        where the last step is a consequence of Lemma \ref{Basics_for_Gamma} \eqref{G2+J2toH2+G1}.
        
    \item $k=3$ and $t=2$, then both $\rho(A')-\frac34$ and $\rho(A'')-\frac12$ are integers and
        $$
        A'\oplus A'' \rightsquigarrow H_2(-1)^{\oplus\rho(A')-\frac34}\oplus J_2(0) \oplus \Gamma_1 \oplus  H_2(-1)^{\oplus\rho(A'')-\frac12}\oplus J_2(0) \rightsquigarrow  H_2(-1)^{\oplus\rho(A)-\frac14}\oplus \Gamma_1,
        $$
        where the last step is a consequence of Lemma \ref{Js_lem}~\eqref{J2^2toH2(-1)}. 
        
\end{itemize}

    \item $\rho(A)-\frac12$ is an integer when either:
\begin{itemize}
    \item $k=0$ and $t=2$, then both $\rho(A')$ and $\rho(A'')-\frac12$ are integers and
        $$
        A'\oplus A'' \rightsquigarrow H_2(-1)^{\rho(A')} \oplus  H_2(-1)^{\oplus\rho(A'')-\frac12} \oplus J_2(0)= H_2(-1)^{\oplus\rho(A)-\frac12}\oplus J_2(0).
        $$
    \item $k=1$ and $t=3$, then both $\rho(A')-\frac14$ and $\rho(A'')-\frac14$ are integers and
        $$
        A'\oplus A'' \rightsquigarrow H_2(-1)^{\rho(A')-\frac14}\oplus \Gamma_1 \oplus  H_2(-1)^{\oplus\rho(A'')-\frac14}\oplus \Gamma_1 \rightsquigarrow  H_2(-1)^{\oplus\rho(A)-\frac12}\oplus \Gamma_1^{\oplus 2}.
        $$
    \item $k=2$ and $t=0$, then both $\rho(A')-\frac12$ and $\rho(A'')$ are integers and
        $$
        A'\oplus A'' \rightsquigarrow H_2(-1)^{\oplus\rho(A')-\frac12}  \oplus J_2(0) \oplus  H_2(-1)^{\oplus\rho(A'')}= H_2(-1)^{\oplus\rho(A)-\frac12}\oplus J_2(0).
        $$
        
    \item $k=3$ and $t=1$, then both $\rho(A')-\frac34$ and $\rho(A'')-\frac34$ are integers and
        $$
        A'\oplus A'' \rightsquigarrow H_2(-1)^{\oplus\rho(A')-\frac34}\oplus J_2(0) \oplus \Gamma_1 \oplus  H_2(-1)^{\oplus\rho(A'')-\frac34}\oplus \Gamma_2 \rightsquigarrow  H_2(-1)^{\oplus\rho(A)-\frac12}\oplus J_2(0),
        $$
        where the last step is a consequence of Lemma \ref{basics_G3_G6} \eqref{G2+G1toH2}.
\end{itemize}

    \item $\rho(A)-\frac34$ is an integer when either: 
\begin{itemize}
    \item $k=0$ and $t=1$, then both $\rho(A')$ and $\rho(A'')-\frac34$ are integers and
        $$
        A'\oplus A'' \rightsquigarrow H_2(-1)^{\oplus\rho(A')} \oplus  H_2(-1)^{\oplus\rho(A'')-\frac34} \oplus \Gamma_2= H_2(-1)^{\oplus\rho(A)-\frac34}\oplus \Gamma_2.
        $$
    \item $k=1$ and $t=2$, then both $\rho(A')-\frac14$ and $\rho(A'')-\frac12$ are integers and
        $$
        A'\oplus A'' \rightsquigarrow H_2(-1)^{\oplus\rho(A')-\frac14}\oplus \Gamma_1 \oplus  H_2(-1)^{\oplus\rho(A'')-\frac12}\oplus J_2(0) \rightsquigarrow  H_2(-1)^{\oplus\rho(A)-\frac34}\oplus \Gamma_2,
        $$
        where the last step is a consequence of Lemma \ref{basics_G3_G6} \eqref{J2+G1toG2}.
    \item $k=2$ and $t=3$, then both $\rho(A')-\frac12$ and $\rho(A'')-\frac14$ are integers and
        $$
        A'\oplus A'' \rightsquigarrow H_2(-1)^{\oplus\rho(A')-\frac12}  \oplus J_2(0) \oplus  H_2(-1)^{\oplus\rho(A'')-\frac14}\oplus \Gamma_1 \rightsquigarrow H_2(-1)^{\oplus\rho(A)-\frac34}\oplus  \Gamma_2,
        $$
        where the last step is a consequence of Lemma \ref{basics_G3_G6} \eqref{J2+G1toG2}.
    \item $k=3$ and $t=0$, then both $\rho(A')-\frac34$ and $\rho(A'')$ are integers and
        $$
        A'\oplus A'' \rightsquigarrow H_2(-1)^{\oplus\rho(A')-\frac34}\oplus J_2(0) \oplus \Gamma_1 \oplus  H_2(-1)^{\oplus\rho(A'')} \rightsquigarrow  H_2(-1)^{\oplus\rho(A)-\frac34}\oplus \Gamma_2,
        $$
        where the last step is a consequence of Lemma \ref{basics_G3_G6} \eqref{J2+G1toG2}.
\end{itemize}
\end{enumerate}
\end{proof}

Now, we are in the position to prove the main result in this section.

\begin{theorem}\label{sufficiency-type1.th}
Let $A$ be a complex square matrix whose {\rm CFC} is a direct sum of Type-I blocks of size at least $3$. Then $X^\top A X = H_2(-1)^{\oplus m}$ is consistent if and only if $m\leq \rho(A)$. 
\end{theorem}

\begin{proof}
The result is a direct consequence of Lemma~\ref{necessary_th} (necessity) and Theorem~\ref{TypeIBlocks_size_geq3_th} (sufficiency). As for the sufficiency, just note that $A\rightsquigarrow H_2(-1)^{\oplus\rho(A)}$ (part (i) in Theorem \ref{TypeIBlocks_size_geq3_th}), or $A\rightsquigarrow H_2(-1)^{\oplus\lfloor\rho(A)\rfloor}\oplus C$, for some $C$ (parts (ii)--(iv) in Theorem \ref{TypeIBlocks_size_geq3_th}). Applying the Elimination law (Lemma \ref{basic_laws_lemma} (iii)), we conclude that $A\rightsquigarrow H_2(-1)^{\oplus \lfloor\rho(A)\rfloor}$, and the Elimination law as well implies $A\rightsquigarrow H_2(-1)^{\oplus m}$, for any integer $m\leq\rho(A)$.
\end{proof}

\section{The condition for consistency is sufficient}\label{sufficiency_sec}

In this section, we prove that the necessary condition of Lemma \ref{necessary_th} is also sufficient when CFC($A$) does not contain blocks of the form $\Gamma_1$ and $\Gamma_2$. This result is included in Theorem \ref{sufficiency.th}, which is the main result of this paper. Before proving it, we state the following lemma, whose proof is straightforward, and that will be used in the proof of Theorem \ref{sufficiency.th}.

\begin{lemma}\label{h2mu+gamma_12_lem}
\begin{enumerate}[{\rm(i)}]
    \item \label{H2+G1^2toH2(-1)}
For $\mu\neq \pm 1$, $H_{2}(\mu)\oplus{\Gamma_{1}}^{\oplus 2} \overset{X_0}{\rightsquigarrow} H_2(-1)$,  with $X_0=
\begin{smat}
 1 & 0 \\
 -1-\mu  & -\frac{2}{\mu -1} \\
 1+\mu  & \frac{1}{\mu -1} \\
 0 & \frac{\mathfrak i}{\mu -1} \\
\end{smat}$.

\item \label{H2+G2toH2(-1)}
For $\mu\neq \pm 1$, $H_{2}(\mu)\oplus{\Gamma_{2}} \overset{X_0}{\rightsquigarrow} H_2(-1)\oplus \Gamma_1$,  with
$X_0=\begin{smat}
 1 & 0 & 1 \\
 -\frac{1}{\mu +1} & 0 & \frac{1}{\mu +1} \\
 0 & 1 & \frac{\mu -1}{\mu +1} \\
 1 & 0 & 0 \\
\end{smat}.$
\end{enumerate}
\end{lemma}

\begin{theorem}\label{sufficiency.th}
If  $A$ is a complex square matrix whose \rm{CFC} has no blocks of either type $\Gamma_1$ or $\Gamma_2$, then for any  skew-symmetric matrix $B$ the equation $X^\top AX=B$ is consistent if and only if $\rank B\leq  2\rho(A)$. In particular, the equation  $X^\top A X = H_2(-1)^{\oplus m}$ is consistent if and only if $m\leq \rho(A)$. 

\end{theorem}
\begin{proof}
Let us start by proving  that  $X^\top A X = H_2(-1)^{\oplus m}$ is consistent if and only if $m\leq \rho(A)$. By Lemma \ref{necessary_th}, the  condition $m\leq \rho(A)$ is necessary. For the sufficiency  it is enough  to show that $A \rightsquigarrow H_2(-1)^{\oplus m}$, with $m=\lfloor\rho(A)\rfloor$, since for a smaller $m$ the result will be a consequence of the Elimination law. 

The CFC of $A$ is of the form $A_0\oplus A_1\oplus A_2$, where 
\begin{itemize}
    \item $A_0$ is an $n_0\times n_0$ direct sum of Type-$0$ blocks; 
    \item $A_1$ is an $n_1\times n_1$ direct sum of Type-I blocks with size $\geq3$;
    \item $A_2$ is an $n_2\times n_2$ direct sum of Type-II blocks.
\end{itemize}
Note that some of $n_0,n_1$, and $n_2$ can be zero. Note also that, by \eqref{additive},
\begin{equation}\label{rhosum}
\rho(A)=\rho(A_0) \oplus \rho(A_1) \oplus \rho(A_2).
\end{equation}
According to, respectively, Theorems~\ref{Js_th}, \ref{Hs.th}, and \ref{TypeIBlocks_size_geq3_th}, we have the following possibilities for $\rho(A_0)$, $\rho(A_2)$, and $\rho(A_1)$:
\begin{enumerate}
\item \begin{enumerate}
    \item $\rho(A_0)$ is an integer and $A_0\rightsquigarrow H_2(-1)^{\oplus\rho(A_0)}$.
    \item $\rho(A_0)-\frac12$ is an integer and $A_0\rightsquigarrow H_2(-1)^{\oplus\rho(A_0)-\frac12}\oplus J_2(0)$.
\end{enumerate}
\item \begin{enumerate}
    \item $\rho(A_2)$ is an integer and $A_2\rightsquigarrow H_2(-1)^{\oplus\rho(A_2)}$.
    \item $\rho(A_2)-\frac12$ is an integer and $A_2\rightsquigarrow H_2(-1)^{\oplus\rho(A_2)-\frac12}\oplus H_2(\mu)$, where $\mu \neq \pm1,0$.
\end{enumerate}  
\item \begin{enumerate}
    \item $\rho(A_1)$ is an integer and $A_1\rightsquigarrow H_2(-1)^{\oplus\rho(A_1)}$.
   \item $\rho(A_1)-\frac14$ is an integer and $A_1\rightsquigarrow H_2(-1)^{\oplus\rho(A_1)-\frac14}\oplus\Gamma_1$.
    \item $\rho(A_1)-\frac12$ is an integer and $A_1\rightsquigarrow H_2(-1)^{\oplus\rho(A_1)-\frac12}\oplus J_2(0)$ or $A_1\rightsquigarrow H_2(-1)^{\oplus\rho(A_1)-\frac12}\oplus \Gamma_1^{\oplus2}$.
    \item $\rho(A_1)-\frac34$ is an integer and $A_1\rightsquigarrow H_2(-1)^{\oplus\rho(A_1)-\frac34}\oplus \Gamma_2$.
\end{enumerate}    
\end{enumerate}

Combining these possibilities the following two cases arise:

\begin{enumerate}[{\tt Case} 1.] 
    \item\label{rho.case1} $0\leq \rho(A)-\lfloor\rho(A_0)\rfloor-\lfloor\rho(A_1)\rfloor-\lfloor\rho(A_2)\rfloor<1$. In this case, $\lfloor\rho(A)\rfloor=\lfloor\rho(A_0)\rfloor+\lfloor\rho(A_1)\rfloor+\lfloor\rho(A_2)\rfloor$. Then 
    $$
    A=A_0 \oplus A_1 \oplus A_2 \rightsquigarrow 
    H_2(-1)^{\oplus\lfloor\rho(A_0)\rfloor}\oplus H_2(-1)^{\oplus\lfloor\rho(A_1)\rfloor}\oplus H_2(-1)^{\oplus\lfloor\rho(A_2)\rfloor}
    =H_2(-1)^{\oplus\lfloor\rho(A)\rfloor}.
    $$
    Note that this case happens when either 1(a)+2(a)+3(a)--(d), or 1(a)+2(b)+3(a)--(b), or 1(b)+2(a)+3(a)--(b) above hold.
    
    \item\label{rho.case2} $1\leq \rho(A)-\lfloor\rho(A_0)\rfloor-\lfloor\rho(A_1)\rfloor-\lfloor\rho(A_2)\rfloor<2$. In this case, $\lfloor\rho(A)\rfloor-1=\lfloor\rho(A_0)\rfloor+\lfloor\rho(A_1)\rfloor+\lfloor\rho(A_2)\rfloor$. Then several subcases arise:
    \begin{itemize}
        \item 1(a)+2(b)+3(c): In this case, either 
        $$
         A=A_0 \oplus A_1 \oplus A_2 \rightsquigarrow 
    H_2(-1)^{\oplus\lfloor\rho(A)\rfloor-1}\oplus
    H_2(\mu)\oplus J_2(0)\rightsquigarrow 
    H_2(-1)^{\oplus\lfloor\rho(A)\rfloor},
    $$
    where the last step is a consequence of Lemma~\ref{H2k+4toH2k_lem} \eqref{H2mu+J2toH2(-1)}, or
    $$
    A=A_0 \oplus A_1 \oplus A_2 \rightsquigarrow 
     H_2(-1)^{\oplus\lfloor\rho(A)\rfloor-1}\oplus H_2(\mu)\oplus\Gamma_1^{\oplus2}\rightsquigarrow H_2(-1)^{\oplus\lfloor\rho(A)\rfloor},
    $$
    where the last step is a consequence of Lemma \ref{h2mu+gamma_12_lem} \eqref{H2+G1^2toH2(-1)}.
        \item 1(a)+2(b)+3(d): In this case,  either
        $$\begin{array}{ccl}
    A=A_0 \oplus A_1 \oplus A_2& \rightsquigarrow &
     H_2(-1)^{\oplus\lfloor\rho(A)\rfloor-1}\oplus H_2(\mu)\oplus\Gamma_2\rightsquigarrow H_2(-1)^{\oplus\lfloor\rho(A)\rfloor}\oplus \Gamma_1\\& \rightsquigarrow& H_2(-1)^{\oplus\lfloor\rho(A)\rfloor},
     \end{array}
    $$
    where the second step is a consequence of Lemma \ref{h2mu+gamma_12_lem} \eqref{H2+G2toH2(-1)}.
    
        \item 1(b)+2(a)+3(c): In this case, either
        $$
        A=A_0 \oplus A_1 \oplus A_2 \rightsquigarrow 
    H_2(-1)^{\oplus\lfloor\rho(A)\rfloor-1}\oplus J_2(0)^{\oplus2}  
    \rightsquigarrow 
    H_2(-1)^{\oplus\lfloor\rho(A)\rfloor},
    $$
    where the last step is a consequence of Lemma \ref{Js_lem}~\eqref{J2^2toH2(-1)},  or
    $$
    A=A_0 \oplus A_1 \oplus A_2 \rightsquigarrow 
     H_2(-1)^{\oplus\lfloor\rho(A)\rfloor-1}\oplus J_2(0)\oplus\Gamma_1^{\oplus2}\rightsquigarrow H_2(-1)^{\oplus\lfloor\rho(A)\rfloor},
    $$
    where the last step is a consequence of Lemma \ref{Basics_for_Gamma} \eqref{J2+G1^2toH2(-1)}.
    
    \item 1(b)+2(a)+3(d): In this case, either
    $$
        A=A_0 \oplus A_1 \oplus A_2 \rightsquigarrow 
    H_2(-1)^{\oplus\lfloor\rho(A)\rfloor-1}\oplus J_2(0)\oplus\Gamma_2
    \rightsquigarrow 
    H_2(-1)^{\oplus\lfloor\rho(A)\rfloor},
    $$
    where the last step is a consequence of Lemma \ref{basics_G3_G6} \eqref{G2+J2toH2+G1}.
    
    \item 1(b)+2(b)+3(a)--(d): In these cases, we have
    $$
    \begin{array}{ccl}
         A=A_0 \oplus A_1 \oplus A_2 &\rightsquigarrow &
    H_2(-1)^{\oplus\lfloor\rho(A)\rfloor-1}\oplus
    J_2(0)\oplus H_2(\mu)\oplus C\\&\rightsquigarrow &
    H_2(-1)^{\oplus\lfloor\rho(A)\rfloor-1}\oplus
    J_2(0)\oplus H_2(\mu)\rightsquigarrow
    H_2(-1)^{\oplus\lfloor\rho(A)\rfloor},
    \end{array}
    $$
    where $C$ is the corresponding leftover from cases 3(a)-(d), and in the second step we remove $C$ by using the Elimination law. The last step is a consequence of Lemma~\ref{H2k+4toH2k_lem} \eqref{H2mu+J2toH2(-1)}.
    \end{itemize}
\end{enumerate}

Now we consider the general case where $B$ is any skew-symmetric matrix. Assume first that $B$ is invertible. According to Lemma~\ref{CFC-Skew_lemma},  ${\rm CFC}(B)=H_2(-1)^{\oplus m}$ with $2m=\rank B$.  By the Canonical reduction law, Lemma \ref{basic_laws_lemma} (\ref{reduction_law}),  $A \rightsquigarrow B$  if and only if $A  \rightsquigarrow H_2(-1)^{\oplus m}$. We have  seen  that $A  \rightsquigarrow H_2(-1)^{\oplus m}$ if and only if $ m\leq\rho(A)$. And  $ m\leq\rho(A)$ if and only if $\rank B\leq  2\rho(A)$. 
%So, to prove that ``$X^\top AX=B$ is consistent if and only if $\rank B\leq  2\rho(A)$" is equivalent  to prove that ``$X^\top AX=H_2(-1)^{\oplus m}$ is consistent if and only if $m\leq \rho(A)$". And the last was   proved above.

If $B$ is singular, then   ${\rm CFC}(B)=H_2(-1)^{\oplus  m}\oplus0_{\ell}$ with  $2 m=\rank B$. By the $J_1(0)$-law, $A\rightsquigarrow H_2(-1)^{\oplus  m}\oplus0_{\ell}$ if and only if $A\rightsquigarrow H_2(-1)^{\oplus m}$. As we have seen above, this holds if and only if $m\leq\rho(A)$. 
\end{proof}

%Theorem\ref{sufficiency.th}  can be stated without making explicit reference to the CFC of $B$, and including the case where it is a singular matrix. This is done in the following theorem, which is the main result of this paper.

%\begin{theorem}
% If $A$ is a complex square matrix whose CFC has no blocks of either type $\Gamma_1$ or $\Gamma_2$ and $B$ is a skew-symmetric matrix, then $X^\top AX=B$ is consistent if and only if $\rank B \leq  2\rho(A)$.
%\end{theorem}

In terms of bilinear forms, as explained in the Abstract and the Introduction, Theorem \ref{sufficiency.th} can be rewritten in the following way.

\begin{theorem}\label{bilinear.th}
Let $\mathbb A$ be a bilinear form on $\CC^n$, and $A\in\CC^{n\times n}$ be some matrix representation of $\mathbb A$. Then, the largest dimension of a subspace of $\CC^n$ such that the restriction of $\mathbb A$ to this subspace is a non-degenerate skew-symmetric bilinear form is $2\lfloor\rho(A)\rfloor$, where $\rho(A)$ is the quantity in \eqref{rho(A)_formula}. 
\end{theorem}

\section{The generic case}\label{generic.sec}

The generic CFC of complex $n\times n$ matrices is the following  (see \cite[Th. 4]{dd11}):
\begin{equation}\label{generic-cfc}
{\rm CFC}_g(n):=\left\{\begin{array}{lc}
    H_2(\mu_1)\oplus\cdots \oplus H_2(\mu_k), & \mbox{\rm if $n=2k$,} \\
     H_2(\mu_1)\oplus \cdots\oplus H_2(\mu_k)\oplus \Gamma_1, & \mbox{\rm if $n=2k+1$,}
\end{array}\right.
\end{equation}
with $\mu_1,\hdots,\mu_k$ being different complex numbers, and all different to $\mu_1^{-1},\hdots,\mu_k^{-1}$ and to $\pm1$ as well.

The following result, which is almost an immediate consequence of Theorem \ref{sufficiency.th}, provides a characterization for Eq. \eqref{maineq} to be consistent for a complex matrix $A$ whose CFC is the one provided in \eqref{generic-cfc}. 

\begin{corollary}
Let $A\in\mathbb C^{n\times n}$ be such that {\rm CFC}$(A)={\rm CFC}_g(n)$ as in \eqref{generic-cfc}, and let $B$ be a skew-symmetric matrix. Then $X^\top AX=B$ is consistent if and only if $\rank B\leq n/2$.
\end{corollary}
\begin{proof}
The case where $n=2k$ in \eqref{generic-cfc} is a particular case of Theorem \ref{sufficiency.th} (note that $\rho(A)=n/4$). 

When $n=2k+1$, Eq. \eqref{generic-cfc} contains a block $\Gamma_1$, which is excluded from the statement of Theorem \ref{sufficiency.th}. However, in this case, $\rho(A)=(2k+1)/4$, so Theorem~\ref{necessary_coro} implies that $\rank B\leq 2\rho(A)= (2k+1)/2=n/2$ is a necessary condition for $A\rightsquigarrow B$. 

If we set $\widetilde A:=H_2(\mu_1)\oplus\cdots\oplus H_2(\mu_k)$, then Theorem \ref{sufficiency.th} implies that $\widetilde A\rightsquigarrow H_2(-1)^{\oplus\frac{k}{2}}$, and the Elimination law (Lemma \ref{basic_laws_lemma}~\eqref{elim_law}) implies that $\widetilde A\oplus\Gamma_1\rightsquigarrow H_2(-1)^{\oplus\frac{k}{2}}$. Therefore, the condition $\rank B\leq n/2$ is also sufficient, since $n/2=k+1/2$ in this case.
\end{proof}

\section{Conclusions and open problems}\label{conclusion.sec}

We have analyzed the consistency of the matrix equation $X^\top AX=B$, for $A\in\CC^{n\times n}$ being a general  matrix, and $B\in\CC^{m\times m}$ being skew-symmetric. In particular, we have obtained a necessary condition for this equation to be consistent. The condition depends on the Canonical Form for Congruence (CFC) of $A$. We have proved that, if the CFC of $A$ does not contain canonical Type-I blocks with size either $1$ or $2$, then this condition is also sufficient. In particular, we have established, for a given $n$, the largest value of $\rank B$ such that the equation is consistent (Theorem \ref{sufficiency.th}). When $A$ is viewed as the matrix of a bilinear form over $\CC^n$, this is equivalent to determining the largest dimension of a subspace of $\CC^n$ such that the restriction of $A$ to this subspace is skew-symmetric.

As a natural continuation of the present work, several lines of research arise:

\begin{itemize}
    \item To address the case where the CFC of $A$ may also contain Type-I blocks with sizes $1$ or $2$. This case deserves a more detailed analysis since, as we have seen, the necessary condition we have obtained for $X^\top AX=B$ to be consistent is not sufficient anymore.
    
    \item To address the consistency of the matrix equation $X^*AX=B$, with $B$ being skew-Hermitian. This will require to use the canonical form for $*$-congruence instead (with $*$ denoting the conjugate transpose).
    
    \item To address the characterization of the consistency of Eq. \eqref{maineq} for $A$ and $B$ being arbitrary square matrices. For this, a natural approach (though problematic, as we will see) is to decompose both $A$ and $B$ in their symmetric and skew-symmetric parts, say $A=A_{\rm sym}+A_{\rm skew}$ and $B=B_{\rm sym}+B_{\rm skew}$. Then if Eq. \eqref{maineq} is consistent, by equating the symmetric and skew-symmetric parts, we end up with $X^\top A_{\rm sym}X=B_{\rm sym}$ and $X^\top A_{\rm skew}X=B_{\rm skew}$, whose right-hand sides are symmetric and skew-symmetric, respectively, so we could apply to them the results in \cite{bcd} and the present paper (respectively). This approach, however, presents relevant obstacles. First, in order to derive a necessary condition for Eq. \eqref{maineq} to be consistent from the necessary condition in this paper (Lemma \ref{necessary_th}), together with the one in \cite{bcd} (namely Theorem 2 in that reference), we should be able to recover the CFC of $A$ from that of $A_{\rm sym}$ and $A_{\rm skew}$. Moreover, we cannot guarantee that the sufficient conditions provided in this paper (namely Theorem \ref{sufficiency.th}) and in \cite{bcd} (namely, Theorem 8 in that reference) are sufficient for Eq. \eqref{maineq} to be consistent for arbitrary $A$ and $B$, since the same $X$ should be a solution for both the symmetric and skew-symmetric parts.
\end{itemize}

\bigskip

\noindent{{\bf Acknowledgments.} This research has been funded by the {\em Agencia Estatal de Investigaci\'on} of Spain through grants PID2019-106362GB-I00/AEI/10.13039/501100011033 and MTM2017-90682-REDT.}


\begin{thebibliography}{99}

\bibitem{bimp} P.\ Benner, B. Iannazzo, B.\ Meini, D.\ Palitta.
\newblock{\em Palindromic linearization and numerical solution of nonsymmetric algebraic $T$-Riccati equations}.
\newblock (2021)  arXiv:2110.03254

\bibitem{benner-palitta} P.\ Benner, D.\ Palitta.
\newblock {\em On the solution of the non-symmetric $T$-Riccati equation}.
\newblock Electron. Trans. Numer. Anal., 54 (2021) 66--88.

\bibitem{benzi-viviani} M.\ Benzi, M.\ Viviani.
\newblock{\em Solving cubic matrix equations arising in conservative dynamics}.
\newblock (2021) arXiv:2111.12373

\bibitem{bcd} A.\ Borobia, R.\ Canogar, F. De Ter\'{a}n.
\newblock{\em On the consistency of the matrix equation $X^\top AX=B$ when $B$ is symmetric}.
\newblock Mediterr. J. Math. 18 (2021). Article number: 40.
%\newblock https://doi.org/10.1007/s00009-020-01656-7

\bibitem{carlitz1954} L. Carlitz.
\newblock{\em Representations by quadratic forms in a finite field.}
\newblock Duke Math. J., 21 (1954) 123--137.

\bibitem{carlitz1954-skew} L. Carlitz.
\newblock {\em Representations by skew forms in a finite field.}
\newblock Arch. Math. (Basel), 5 (1954) 19--31. 

\bibitem{semaj} F. De Ter\'an.
\newblock{\em Canonical forms for congruence of matrices: a tribute to H. W. Turnbull and A. C. Aitken}.
\newblock SeMA J., 73 (2016) 7-16.

\bibitem{dd11} F. De Ter\'an, F. M. Dopico.
\newblock {\em The solution of the equation $XA + AX^T = 0$ and its
application to the theory of orbits}.
\newblock Linear Algebra Appl., 434 (2001) 44--67.


\bibitem{horn-johnson} R.\ A.\ Horn, C.\ R.\ Johnson.
\newblock{\em Matrix Analysis}.
\newblock Cambridge University Press, 2nd Ed. New York, 2013.

\bibitem{hs2006} R.\ A.\ Horn, V.\ V.\ Sergeichuk.
\newblock{\em Canonical forms for complex matrix congruence and $*$-congruence}. \newblock Linear Algebra Appl., 416 (2006) 1010--1032.

\bibitem{ikramov14} Kh. D. Ikramov.
\newblock{\em On the solvability of a certain class of quadratic matrix equations.}
\newblock Dokl. Math. 89 (2014) 162--164.

\bibitem{wedderburn1921} J. H. M. Wedderburn.
\newblock{\em The automorphic transformation of a bilinear form.}
\newblock Ann. of Math. 2, 23 (1921) 122--134.

\end{thebibliography}
\end{document}